\documentclass{amsart}
\usepackage[margin=1in]{geometry}
\usepackage{amsthm}
\usepackage{amssymb}
\usepackage{enumitem}
\usepackage{mathtools}
\usepackage{xcolor}
\usepackage{hyperref}
\usepackage{verbatim}

\usepackage{mathtools}
\usepackage{amsmath,accents}

\mathtoolsset{showonlyrefs}

\usepackage{hyperref}
\hypersetup{
    colorlinks,
    citecolor=blue,
    filecolor=blue,
    linkcolor=blue,
    urlcolor=blue
}

\title[Ancient and expanding spin ALE Ricci flows]{Ancient and expanding spin ALE Ricci flows}
\author[Isaac M. Lopez]{Isaac M. Lopez}
\author[Tristan Ozuch]{Tristan Ozuch}
\address{MIT, Department of Mathematics, 77 Massachusetts Avenue, Cambridge, MA 02139, USA.}
\email{imlopez@mit.edu, ozuch@mit.edu}
\date{}
\numberwithin{equation}{section}

\begin{document}
\newcommand{\vol}{\text{Vol}}
\newcommand{\euc}{\text{euc}}
\newcommand{\Ric}{\mathrm{Ric}}
\newcommand{\Scal}{\text{Scal}}
\newcommand{\Hess}{\text{Hess}}
\newcommand{\ric}{\text{Ric}}
\newcommand{\C}{\mathcal{C}}
\newcommand{\R}{\mathbb{R}}
\newcommand{\D}{\Delta}
\newcommand{\bn}{\bigskip \noindent}
\newcommand{\T}{\mathcal{T}}
\newcommand{\ov}{\overline}
\newcommand{\osc}{\text{osc}}
\newcommand{\E}{\mathcal{E}}
\newcommand{\N}{\mathbb{N}}
\newcommand{\cpt}{\subset\subset}
\newcommand{\st}{\text{ } | \text{ }}
\newcommand{\Z}{\mathbb{Z}}
\newcommand{\andd}{\text{ } \text{ } \text{ and } \text{ } \text{ }}
\newcommand{\hess}{\text{Hess}}
\newcommand{\Rm}{\text{Rm}}
\renewcommand{\div}{\operatorname{div}}
\renewcommand{\L}{\mathcal{L}}
\newtheorem{theorem}{Theorem}[section]
\newtheorem{corollary}[theorem]{Corollary}
\newtheorem{lemma}[theorem]{Lemma}
\newtheorem{definition}[theorem]{Definition}
\newtheorem{question}[theorem]{Question}
\newtheorem{proposition}[theorem]{Proposition}
\newtheorem{conjecture}[theorem]{Conjecture}
\newtheorem{remark}{Remark}
\newtheorem{example}[theorem]{Example}
\newcommand{\tr}{\operatorname{tr}}
\newcommand{\rel}{\text{ } \mathrm{rel}}
\renewcommand{\r}{\rightarrow}
\newcommand{\scal}{\mathrm{Scal}}
\newcommand{\W}{\mathcal{W}}
\newcommand{\F}{\mathcal{F}}
\renewcommand{\N}{\mathcal{N}}
\renewcommand{\geq}{\geqslant}
\renewcommand{\leq}{\leqslant}
\newcommand{\ALE}{\mathrm{ALE}}
\newcommand{\ADM}{\mathrm{ADM}}

\maketitle 

\begin{abstract}
    We classify spin ALE ancient Ricci flows and spin ALE expanding solitons with suitable groups at infinity. In particular, the only spin ancient Ricci flows with groups at infinity in $SU(2)$ and mild decay at infinity are hyperkähler ALE metrics. The main idea of the proof, of independent interest, consists in showing that the large-scale behavior of Perelman's $\mu$-functional on any ALE orbifold with non-negative scalar curvature is controlled by a renormalized $\lambda_{\ALE}$-functional related to a notion of weighted mass.
\end{abstract}

\setcounter{tocdepth}{1}

\tableofcontents

\section{Introduction}

Understanding Ricci flow in dimension $4$ is a major avenue for research towards applications to questions in $4$-dimensional topology, where the main concern is the homeotype of \textit{spin} manifolds. The main challenge consists in classifying in some way the topological surgeries corresponding to the singular times of the flow. Rescaling a Ricci flow close to a singular time yields \textit{ancient solutions} of the Ricci flow whose classification is central. \textit{Expanding solitons} may then be used to resolve finite-time singularities and restart the flow. In this article, we show that there is an intriguing rigidity for an important class of spin ancient and expanding flows, under topological assumptions.

\subsection{Main motivations}

\subsubsection{Orbifold singularities along Ricci flow}

Compared to dimension $3$ and lower, a new type of singularity from dimension $4$ is that of (isolated) \textit{orbifold singularities} modeled on $\mathbb{R}^4/\Gamma$ with $\Gamma\subset SO(4)$ acting freely on $\mathbb{S}^3$, the only Ricci-flat cones in dimension $4$. They are the most severe types of singularities of specific limits of infinite-time blow-downs of immortal Ricci flows, and of the singular-time blow-ups limit \cite{bamler-21-1, bamler-2023,
bamler-2021-3}.  Therefore, a major question has been:
\begin{question}
    How do orbifold singularities form or are resolved along Ricci flow?
\end{question}
Blowing up (or down) such singularity formation (or resolution) at specific scales leads to \textit{ancient} Ricci flows with tangent soliton $\mathbb{R}^4/\Gamma$ at $-\infty$. Ricci-flat ALE metrics are obvious examples of such ancient flows.

\begin{question}
    What are the ancient Ricci flows with tangent soliton $\mathbb{R}^4/\Gamma$ at $-\infty$?
\end{question}

In the noncollapsed Einstein context and along Ricci flows with bounded scalar curvature, orbifold singularities are the only possible singularities \cite{Simon, Bamler-Zhang}. In this context, it is known that these singularities must be related to Ricci-flat ALE metrics. Additionally, in the article \cite{deruelle-ozuch-2024}, large classes of ancient and immortal Ricci flows are constructed from the interactions of orbifold singularities and Ricci-flat ALE metrics. See also \cite{brendle-kapouleas} for a similar instance of orbifold singularities resolved along Ricci flow.

In the present article, we show that hyperkähler ALE metrics are the only \textit{spin} ancient Ricci flows with tangent soliton $\mathbb{R}^4/\Gamma$, $\Gamma\subset SU(2)$, decaying mildly at infinity. Similarly, we rule out the existence of \textit{spin} ALE expanding solitons with nontrivial group in $SU(2)$ at infinity, which are natural candidates to resolve orbifold singularities while preserving the spin condition. Similar statements hold in higher dimensions.

\subsubsection{Functionals on ALE manifolds.}

Our work is part of a broader effort of defining useful asymptotic quantities on noncompact manifolds, see for instance \cite{dahl-kroencke-mccormick, kroencke-yudowitz} in other contexts, and the discussion below on ALE manifolds. Theorem \ref{large-scale-behavior-of-mu} opens the question of whether these functionals can, in some sense, be traced back to Perelman's $\mu$-functional, which might explain their satisfying behavior.

An \textit{ALE manifold} is a complete manifold asymptotic to a flat metric $(\mathbb{R}^n/\Gamma,d_e)$ for $\Gamma\subset SO(n)$ finite acting freely on $\mathbb{S}^{n-1}$. It is of order $\beta>0$ if the metric decays to the flat Euclidean metric like $d_e(0,\cdot)^{-\beta}$, see Definition \ref{def:ALE}. The proof of our classification of spin ALE ancient Ricci flows relies on the behavior of specific functionals on such manifolds. 

The most notable quantity related to the large-scale properties of ALE manifolds is the ADM mass, a central quantity in general relativity. It will be denoted $\frak{m}$ here. Other quantities detecting large-scale properties of the manifold have recently been introduced, motivated by the study of Ricci flow. They have the advantage of being defined on larger classes of metrics than mass.
\begin{itemize}
    \item A renormalized version of Perelman's $\lambda$, which we will denote $\lambda_{\ALE}^0$, was introduced in \cite{haslhofer}.
    \item It was later refined as a functional $\lambda_{\ALE}$, which is analytic in useful spaces, and whose gradient flow is Ricci flow, \cite{deruelle-ozuch-2020}. It was used to study the stability of Ricci-flat ALE metrics in \cite{deruelle-ozuch-2021}.
    \item It was then noticed that $-\lambda_{\ALE}$ is a weighted version of the ADM mass $\frak{m}$ in \cite{baldauf2022spinors}, denoted $\frak{m}_f$ for a natural choice of $f$. The positivity of $\frak{m}_f$ on AE manifolds on which the classical positive mass theorem holds has recently been proven in \cite{baldauf2022spinors, chu-zhu, law-lopez-santiago}.
    \item Another weighted \textit{spinorial} quantity encompassing all of the above as well as the spinorial energy of \cite{Ammann-Weiss-Witt} was introduced in \cite{baldauf-ozuch-2}.
    \item Perelman's $\mu$-functional is also well-defined on ALE manifolds, and its large-scale asymptotics are analyzed in the present article.
\end{itemize}

\subsubsection{Stability, positive mass and uniqueness of Ricci flow}

Inspired by partial analogies with minimal surfaces and mean curvature flows, questions about Ricci flows around Ricci-flat cones can be found in \cite[Section 10]{Feldman-Ilmanen-Knopf} and in \cite{hall-haslhofer-siepmann}. In dimension $4$, Ricci-flat cones are necessarily of the form $\mathbb{R}^4/\Gamma$ for $\Gamma\subset SO(4)$ acting freely on $\mathbb{S}^3$ since positively curved Einstein $3$-manifolds are space forms. In this context, what is sometimes called \emph{Ilmanen's conjecture} predicts a relationship between:
\begin{enumerate}
    \item the ``stability'' of $\mathbb{R}^4/\Gamma$, here understood as being a local maximizer of the $\lambda_{\ALE}$-functional among ALE metrics with nonnegative scalar curvature and asymptotic cone $\mathbb{R}^4/\Gamma$,
    \item positive mass theorems for manifolds asymptotic to $\mathbb{R}^4/\Gamma$, and
    \item Ricci flows coming out of $\mathbb{R}^4/\Gamma$ their uniqueness. Expanding solitons have been proven to be the only possible Ricci flows coming out of cones under specific curvature assumptions in \cite{deruelle-schulze-simon, chan-lee-peachey}
\end{enumerate}
After the initial treatment of \cite{hall-haslhofer-siepmann}, the above relationships--especially between the first two points--have been uncovered in \cite{deruelle-ozuch-2020,baldauf2022spinors}. The cones $\mathbb{R}^4/\Gamma$ for $\Gamma\subset SU(2)$ (1) are $\lambda_{\ALE}$-stable, and (2) the positive mass theorem holds on them, \emph{if the topology is assumed to be spin and compatible with that of the asymptotic cone} $\mathbb{R}^4/\Gamma$, \cite{deruelle-ozuch-2020,baldauf2022spinors}. 
In the present article, we obtain an essentially complete answer regarding the last point as well. Theorems \ref{thm: exp ALE} and \ref{large-scale-behavior-of-mu} state that on $\lambda_{\ALE}$-stable fillings, e.g. the spin desingularizations of $\mathbb{R}^4/\Gamma$ with $\Gamma\subset SU(2)$, (3) there is no expanding Ricci flow coming out of $\mathbb{R}^4/\Gamma$. 

\begin{remark}
    The stability and the existence of Ricci flows is very much related to the topology of the manifold inside the Ricci-flat cone $\mathbb{R}^4/\Gamma$ for $\Gamma\subset SU(2)$--it cannot be solely read off from the cone. Indeed, while spin fillings are stable, the is not the case of other fillings. This can be seen from ALE metrics on $O(-k)$ for $k\geqslant 3$ considered with \emph{the opposite orientation}: they are asymptotic to $\mathbb{R}^4/\mathbb{Z}_k$ for $\mathbb{Z}_k\subset SU(2)$, but (1) they are $\lambda_{ALE}$-unstable by \cite{deruelle-ozuch-2020}, (2) the positive mass theorem does not hold on $O(-k)$ by \cite{lebrun}, and (3) there are expanding Ricci flows on $O(-k)$ coming out of $\mathbb{R}^4/\mathbb{Z}_k$ by \cite{Feldman-Ilmanen-Knopf}.
\end{remark}

\subsection{Main results}

In dimension $4$, manifolds which are not spin are classified up to homeomorphisms and this is not the case for spin $4$-manifolds. Spin $4$-manifolds are characterized by having an even intersection form. 

We classify expanding and ancient ALE Ricci flows under a topological assumption: they admit a spin structure compatible with their infinity and have specific groups at infinity. We will call these \textit{spin ALE} manifolds.

\subsubsection{Classification of spin ALE ancient Ricci flows}

Ancient Ricci flows model the formation of singularities along Ricci flow, making their classification crucial. Three-dimensional $\kappa$-noncollapsed ancient solutions in are classified obtained in \cite{Brendle2020,
Ang-Bre-Das-Ses,Bre-Das-Ses}. In higher dimensions, a similar classification can be found in \cite{Bre-Das-Naf-Ses} under stronger curvature assumptions.  
\\

By \cite{yu-li}, the ALE condition and the order $\beta>0$ are preserved by Ricci flow. This allows us to define an ALE Ricci flow of order $\beta$. We will call such a flow \textit{ancient} if it is defined on $( -\infty, T ] $ for $T\in \mathbb{R}$.

\begin{theorem} \label{thm: anc ALE}
    Suppose \((M^n,g_t)_{t\in( -\infty, T ]}\) is an ancient spin ALE Ricci flow of order \(\beta\) with group at infinity in $SU(2)$ and $\beta>\frac{4}{3}$ if $n=4$, or satisfying the assumptions of \cite[Theorem 5.1]{dahl} in other dimensions, and where either $3\leq n\leq 6$ and $\beta > \frac{n}{3}$ or $n\geq 7$ and $\beta > \frac{n-2}{2}$. 
    
    Then \((M^n,g)\) admits a parallel spinor, and is in particular Ricci-flat.
\end{theorem}

In dimension $4$, this means that any spin ALE ancient Ricci flow with group at infinity in $SU(2)$, and of order $>\frac{4}{3}$, must be hyperkähler. Note that given a group $\Gamma\subset SU(2)$, the spin assumption still allows infinitely many diffeotypes (e.g. by connected sum with arbitrarily many $\mathbb{S}^2\times \mathbb{S}^2$), while being hyperkähler only allows one.

\begin{example}
    Any ancient Ricci flow ALE of order $>\frac{4}{3}$ with topology $O(-2)=T^*\mathbb{S}^2$ must be homothetic to the static Eguchi-Hanson metric, which is ALE of order $4$. 
    
    Similarly, any ancient ALE Ricci flow diffeomorphic to a minimal resolution of $\mathbb{C}^2/\Gamma$ and of order $>\frac{4}{3}$ must be one of the hyperkähler metrics classified in \cite{kronheimer-hyperkahler-quotients, Kronheimer-Torelli-theorem}, which are of order at least $4$.

    Theorem \ref{thm: anc ALE} also applies to $M^n$ equipped with an ALE Calabi-Yau metric such as the Calabi metric \cite{calabi-1}, or the examples of \cite{joyce} and \cite{tian-yau}. The corresponding ALE Calabi-Yau metrics are of order at least $n$.
\end{example}

\begin{remark}
    Towards understanding Ricci flow on spin $4$-manifolds, Theorem \ref{thm: anc ALE}, up to a minor decay assumption, states that if a Ricci flow develops or resolves a singularity $\mathbb{R}^4/\Gamma$ for $\Gamma\subset SU(2)$, it must bubble-off a hyperkähler metric.
\end{remark}

\subsubsection{Classification of spin ALE expanding solitons} 

In order to restart a Ricci flow at a finite-time singularity, it has been proposed to use expanding solitons in the case of conical singularities (such as orbifold singularities). For instance, in \cite{Feldman-Ilmanen-Knopf}, a Ricci flow composed of a shrinking soliton up to a singular time, then desingularized by an expanding soliton, is presented. See also \cite{Gianniotis-Schulze} for other resolutions of singularities using expanding solitons, \cite{deruelle-2016} for large classes of expanding solitons, and \cite{conlon-deruelle, conlon-deruelle-sun} for constructions and classifications of expanding solitons in the Kähler case.
\\

Expanding solitons are often considered to be abundant, flexible, and to exist in large families on a variety of topologies. We instead obtain strong restrictions for them on spin ALE manifolds. 

\begin{theorem} \label{thm: exp ALE}
    Suppose \((M^n,g)\) is a spin ALE expanding soliton and with group at infinity in $SU(2)$ if $n=4$, or satisfying the assumptions of \cite[Theorem 5.1]{dahl} in other dimensions.
    
    Then \((M^n,g)\) is a flat, Gaussian expanding soliton.
\end{theorem}

\begin{example}
    This shows that there cannot be any ALE expanding soliton on $O(-2)=T^*\mathbb{S}^2$, on the minimal resolution of any $\mathbb{C}^2/\Gamma$ for $\Gamma\subset SU(2)$, or on $M^n$ equipped with an ALE Calabi-Yau metric. 
\end{example}

By contrast, there exist ALE expanding solitons of infinite order on $O(-k)$ for $k\geqslant 3$, \cite{Feldman-Ilmanen-Knopf}; however, the groups at infinity are in $U(2)\backslash SU(2)$, hence Theorem \ref{thm: exp ALE} does not apply. By \cite{deruelle-ozuch-2020}, $\lambda_{\ALE}(g)>0$ for such metrics, and they have larger $\nu$-functional than their asymptotic cone.

\begin{remark}
     Theorem \ref{thm: exp ALE} shows that it is impossible to resolve a $\mathbb{R}^4/\Gamma$ singularity with $\Gamma\subset SU(2)$ by an expanding soliton while preserving the spin condition. The examples of \cite{Feldman-Ilmanen-Knopf} show that it is possible for specific subgroups of $U(2)\backslash SU(2)$.
\end{remark}

\subsubsection{Large-scale behavior of Perelman's $\mu$-functional on ALE manifolds}

Ancient and expanding ALE Ricci flows can be seen as Ricci flows whose ``initial data'' is a flat cone $\mathbb{R}^n/\Gamma$. By the monotonicity of Perelman's $\mu$-functional, this should intuitively imply that the $\mu$-functional of said flows is \textit{larger} than that of $\mathbb{R}^n/\Gamma$. This is made rigorous in Propositions \ref{prop:nu exp} and \ref{prop: ineq nu ancient} below. 

Our strategy to prove Theorems \ref{thm: anc ALE}  and Theorem \ref{thm: exp ALE} is to reach a contradiction by estimating the $\mu$-functional on ancient or expanding ALE Ricci flows satisfying our topological assumptions.
\\

Recall that on a compact manifold $(M,g)$, one has the asymptotic expansion
$$ \mu(g,\tau) = \tau \lambda(g) + o(\tau) $$
as $\tau\to\infty$. That is, the large-scale behavior of Perelman's $\mu$-functional is dictated by the $\lambda$-functional.
\\

Our classification of spin ALE Ricci flows relies on the following theorem which makes explicit the large-scale behavior of Perelman's \(\mu\)-functional on ALE manifolds. It is this time controlled by the $\lambda_{\ALE}$-functional introduced in \cite{deruelle-ozuch-2020}.

\begin{theorem} \label{large-scale-behavior-of-mu}
    Suppose \((M^n,g)\) is an ALE manifold of order \(\beta\) asymptotic to $\mathbb{R}^n/\Gamma$, where $3\leq n\leq 6$ and $\beta > \frac{n}{3}$ or $n\geq 7$ and $\beta > \frac{n-2}{2}$, and that $\scal_g\geqslant 0$. Then
    \begin{equation}\label{eq:expansion mu}
        \mu(g,\tau)\, \leq \,\mu(\mathbb{R}^n/\Gamma) + \frac{\tau}{(4\pi \tau)^{\frac{n}{2}}}|\Gamma| \cdot\lambda_{\mathrm{ALE}}(g) + O(\tau^{\gamma}),
    \end{equation}
    where \(\gamma<1-\frac{n}{2}\) is as in Lemma \ref{noncompact-asymptotics-for-arbitrary-ALE}.
    
    In particular, if
    \begin{itemize}
        \item $g$ has positive scalar curvature,
        \item $\Gamma\subset SU(2)$ if $n=4$, or satisfying the assumptions of \cite[Theorem 5.1]{dahl} if $n>4$, and
        \item $M^n$ has a spin structure compatible with the spin structure of $\mathbb{R}^n/\Gamma$ in the sense of the assumptions of \cite[Theorem 5.1]{dahl},
    \end{itemize}
    then $\lambda_{\mathrm{ALE}}(g) < 0$, and by \eqref{eq:expansion mu}, for sufficiently large \(\tau\), we have 
    \begin{equation}
        \mu(g,\tau) < \mu(\mathbb{R}^n/\Gamma).
    \end{equation}
\end{theorem}
\begin{remark}
    The terms in \eqref{eq:expansion mu} are all (parabolically) scale-invariant: for any $s>0$, $\mu(s\,g,s\,\tau) = \mu(g,\tau)$, $\mu(s\,\mathbb{R}^n/\Gamma) = \mu(\mathbb{R}^n/\Gamma)$, and $\lambda_{\ALE}(s\,g) = s^{\frac{n}{2}-1}\lambda_{\ALE}(g)$, see \cite[Proposition 3.12]{deruelle-ozuch-2020}.
\end{remark}

\begin{remark}
    As proven in \cite[Theorem 6.2]{dahl}, in dimension $4$, the assumption on the compatibility of the spin structure of the manifold can actually be replaced by the value of the signature of the manifold. 
\end{remark}

\begin{remark}
    The first part of Theorem \ref{large-scale-behavior-of-mu} also has an extension to manifolds without nonnegative scalar curvature, see Remark \ref{arbitrary-ALE-metric-remark}; however, this is not needed to prove Theorems \ref{thm: anc ALE} and \ref{thm: exp ALE}.
\end{remark}

It is important to mention that the conclusion that \(\mu(g,\tau) < \mu(\mathbb{R}^n/\Gamma)\) for manifolds satisfying the above conditions depends on the fact that \(\lambda_{\mathrm{ALE}}(g)<0\) for such manifolds. In \cite[Theorem 3.4]{yu-li}, Li proves that \(\mu(g,\tau)\r 0\) as \(\tau \r +\infty\), so Theorem \ref{large-scale-behavior-of-mu} can be viewed as a quantitative version of Li's result when applied to AE manifolds. In the ALE case with $\Gamma\neq \{\operatorname{Id}\}$, since $\mu(\mathbb{R}^n/\Gamma)< 0$, even proving that $\liminf_{\tau\to +\infty}\mu(g,\tau) \leqslant \mu(\mathbb{R}^n/\Gamma)$ is not that simple.
\\

The proof of Theorem \ref{large-scale-behavior-of-mu} relies on a careful construction of a test function for Perelman's $\mathcal{W}$-functional. The argument is delicate and requires the use of a very specific \textit{radial} gauge at infinity for ALE metrics from \cite{x-ray-transform}. The negativity of $\lambda_{\ALE}$ on spin ALE manifolds with suitable group at infinity is an extension of an argument of \cite{deruelle-ozuch-2020} relying on the ALE positive mass theorems of \cite{nakajima,dahl}.

\subsection{Organization of the article} In Section \ref{background-on-geometric-functionals}, we briefly review some relevant facts pertaining to Ricci flow on AE manifolds proved in \cite{yu-li} as well as the renormalized \(\lambda\)-functional \(\lambda_{\mathrm{ALE}}\) introduced in \cite{deruelle-ozuch-2020}. In Section \ref{large-scale-behavior}, we prove Theorem \ref{large-scale-behavior-of-mu}. In Section \ref{classification-of-ALE-expanding-solitons-and-ancient-flows}, we prove Theorems \ref{thm: anc ALE} and \ref{thm: exp ALE}. We discuss and introduce other dynamical quantities in Section \ref{sec:further} and conclude with some open questions.

\subsection{Notation and conventions}
We will freely use the following notation and conventions throughout the article.
\begin{itemize}
    \item $dV_g$ denotes the volume form with respect to a given Riemannian metric $g$, and $dA_g$ denotes the volume form of a given sphere with outward unit normal $\nu$, i.e. $dA_g = \iota_\nu dV_g$.
    \item $g_e$ denotes the standard Euclidean metric, and $\nabla^e=\nabla^{g_e}$, $dV_e=dV_{g_e}$, etc.
    \item We define $\Gamma_\tau:=\exp(-r(x)^2/(8\tau))$, where $r(x)$ is defined in \S\ref{ALE-review}.
    \item We will write $x \lesssim y$ to mean that $x\leq cy$ for some constant $c>0$ which is independent of $r(x)=|x|$ and $\tau$. More precisely, $c$ will only depend on the constants $C_k$ in Definition \ref{def:ALE}. We will also write $x=O(y)$ to mean that $|x|\lesssim y$.
\end{itemize}

\subsection{Acknowledgments}

The authors would like to thank Zilu Ma for explaining the argument of Proposition \ref{prop: ineq nu ancient} to them, and Alix Deruelle for suggesting the references used in Proposition \ref{prop:nu exp}. I.M.L. was partially supported by the MIT Undergraduate Research Opportunities Program (UROP), and was mentored by T.O.

\section{Background on geometric functionals on ALE manifolds} \label{background-on-geometric-functionals}

\subsection{ALE manifolds, mass, and Ricci flows} \label{ALE-review}
Let us first recall the definition of an ALE manifold.

\begin{definition}[ALE manifold]\label{def:ALE}
    A Riemannian manifold \((M^n,g)\) is \emph{asymptotically locally Euclidean (ALE)} of order \(\beta > 0\) if there exists a compact set \(K \subseteq M\), a radius \(R>1\), a subgroup \(\Gamma\) of \(SO(n)\) acting freely on \(\mathbb{S}^{n-1}\) and a diffeomorphism \(\Phi: (\R^n/\Gamma)\setminus B_e(0,R) \mapsto M\setminus K=M_\infty\) such that, if we denote by \(g_e\) the Euclidean metric on \(\R^n/\Gamma\), then for each \(k\in \{0,1,2,3\}\), there exists $C_k>0$ such that
    \begin{equation}
        \rho^k|\nabla^{g_e,k}(\Phi^*g - g_e)|_e = C_k\rho^{-\beta}
    \end{equation}
    on \((\R^n/\Gamma)\setminus B_e(0,R)\), where \(\rho = d_e(\cdot,0)\).
\end{definition}
The degree of regularity is chosen for convenience, it could easily be refined in most of our estimates. Our main applications will be to very smooth objects.

We fix a smooth positive function \(r(x)\geqslant 1\) on \(M\) such that \(r(x)=|\Phi(x)|\) when \(x\in M_\infty\). We also identify \(x\in M_\infty\) with \(\Phi(x)\in \R^n/\Gamma\) without explicitly referring to \(\Phi\), and we denote by $B_{r_0}$ the set of $x\in M$ such that $r(x)\leq r_0$, and by $S_{r_0}$ the set of $x\in M$ such that $r(x)=r_0$.

The notion of an ALE manifold is a more general one than that of an AE (asymptotically Euclidean) manifold. Ricci flow on AE manifolds is well described in \cite{yu-li} and its results extend to the ALE case. A principal goal of this article is to better understand Ricci flow on the larger class of ALE manifolds.

The mass of an ALE metric is well-defined on the classical space $\mathcal{M}_\beta$ of ALE spaces of order $\beta>\frac{n-2}{2}$ with integrable scalar curvature, of \cite{lee1987yamabe} and \cite{bartnik1986mass}, on which it can be written as
\begin{equation}
    \mathfrak{m}(g) = \lim_{R \r \infty} \int_{\{r=R\}} \langle \div_{e}(g-g_e) - \nabla^{e} \tr_{e}(g-g_e), \nu \rangle_{e} dA_{e}.
\end{equation}
Outside this space, however, there is no guarantee that the mass is defined, and the same goes for $\int \scal_g dV_g$. Such subtleties present significant challenges in the proof of Theorem \ref{large-scale-behavior-of-mu}, as described in \S\ref{large-scale-behavior}.

\begin{remark}
    The classical positive mass theorem \cite{schoen1979PMT1,witten1981new} does not hold for ALE spaces \cite{lebrun}, even when restricted to the class of ALE expanding solitons decaying exponentially fast at infinity, \cite{Feldman-Ilmanen-Knopf}.
\end{remark}

An important fact about Ricci flow on ALE manifolds is that \emph{the mass remains constant along the flow.} This is proven in \cite[Theorem 2.2]{yu-li} in the AE case, but the proof directly extends to ALE metrics since an ALE end is covered by a $\Gamma$-invariant AE end.

\begin{theorem}[\cite{yu-li}, Theorem 2.2, see also \cite{dai-ma}]
    Suppose $(M^n,g(t))$ is a solution to the Ricci flow with bounded curvature on $M\times [0,T]$ and $(M,g(0))$ is ALE of order $\beta >0$. Then
    \begin{enumerate}
        \item The ALE condition is preserved with the same ALE coordinates and order.
        \item If $\beta > \frac{n-2}{2}$ and $\scal_g \in L^1(dV_g)$, then the mass remains constant under the flow. 
    \end{enumerate}
\end{theorem}

This lets us define a reasonable notion of ALE Ricci flow. 
\begin{definition}\label{def:ALErf}
    Let $(M,g_t)_{t\in I}$ for $I\subset \mathbb{R}$ an interval. Then, we say that $(M,g_t)_t$ is an ALE Ricci flow of order $\beta>0$ if for every $t\in I$, $(M,g_t)$ is an ALE manifold of order $\beta$ in the sense of Definition \eqref{def:ALE}. 
\end{definition}

Note that the constants $C_k(t)$ in Definition \ref{def:ALE}, and the radius $R(t)$ are not supposed to be controlled. Still, \cite{yu-li} ensures that the resulting Ricci flow has \textit{bounded curvature within each compact time interval} as defined in \eqref{eq:bdd curv finite time}, and that the order $\beta$ is indeed preserved. 

\emph{Weighted Hölder spaces} appear frequently in asymptotic geometry, and some of their properties will be helpful in some of the proofs.

\begin{definition}[Weighted Hölder space]
Let $(M^n,g)$ be an ALE manifold with asymptotic coordinates $x$ on $M_\infty$, and recall the definition of $r(x)$ from Definition \ref{def:ALE}. For $0 < \alpha < 1$, $k \in \mathbb{N}_0$ and $\beta \in \R$, the \emph{weighted H\"older space} $\C^{k,\alpha}_{\beta}$ is the space of $\C^k$ functions $u: M \to \R$ for which the norm

    \begin{equation}
        \lVert u \rVert_{\C^{k,\alpha}_\beta(M)} := \sum_{0 \leq i \leq k} \left( \sup_{x \in M_\infty} \frac{|\nabla^i u(x)|}{r^{\beta-i}} \right) +
        \sup_{x \in M_\infty} \frac{[\nabla^k u]_{\mathcal{C}^\alpha(B_{r/2}(x))}}{r^{\beta-(k+\alpha)}}
    \end{equation}
    is finite, where $B_{r/2}(x)$ is the metric ball of radius $\frac{r}{2}$ centered at $x$ and
    \begin{align}
        [\nabla^k u]_{\mathcal{C}^\alpha(B_{r/2}(x))} &:= \sup_{y,z \in B_{\frac{r}{2}}(x)} \frac{|\nabla^k u(y) - \nabla^k u(z)|}{|y-z|^\alpha}.
    \end{align}
\end{definition}
Note that if $u \in \C^{k,\alpha}_\beta(M)$, then $u = O(r^\beta)$ as $r \to \infty$. Below are some useful properties of weighted Hölder spaces, which are straightforward to check using the definition.
\begin{enumerate}
    \item If $u \in \C^{k,\alpha}_\beta(M)$ and $v\in \C^{k,\alpha}_{\beta'}(M)$, then $u+v\in \C^{k,\alpha}_{\max(\beta,\beta')}(M)$ and $uv \in \C^{k,\alpha}_{\beta+\beta'}(M)$.
    \item If $u\in \C^{k,\alpha}_\beta(M)$, then $\nabla^j u \in \C^{k-j,\alpha}_{\beta-j}(M)$.
\end{enumerate}

\subsection{$\lambda_{\ALE}$ functional}
\begin{definition}[A first renormalized Perelman's functional, \cite{haslhofer}]
    Let $(M^n,g)$ be an ALE manifold. The \emph{$\mathcal{F}_{\ALE}$-energy} is defined as
    \begin{equation}
        \mathcal{F}_{\ALE}(u,g) := \int_M (4|\nabla u|^2 + \scal_g u^2)dV_g,
    \end{equation}
    where $u\in \C^\infty(M)$, $u-1=O(r^{-\beta})$, and $|\nabla u| \in L^2(dV_g)$. The \emph{$\lambda^0_{\ALE}$-functional} is
    \begin{equation}
        \lambda^0_{\ALE}(g):= \inf\{\mathcal{F}_{\ALE}(u,g) \st u\in \C^\infty(M), u-1=O(r^{-\beta}), |\nabla u| \in L^2(dV_g)\}.
    \end{equation}
\end{definition}
\begin{remark}
    This quantity is only finite on manifolds with integrable scalar curvature.
\end{remark}

There is always a positive minimizer of $\mathcal{F}_{\ALE}$ for ALE manifolds of non-negative scalar curvature, as guaranteed by the following proposition. The existence of such a minimizer is crucial for our large-scale estimates for the $\mu$-functional.

\begin{proposition}[\cite{deruelle-ozuch-2020}, Proposition 1.12; \cite{baldauf2022spinors}, Theorem 2.17; \cite{haslhofer}, Theorem 2.6] \label{energy-minimizers-for-positive-scalar-curvature}
    Let $(M^n,g)$ be an ALE manifold with non-negative scalar curvature, asymptotic to $\R^n/\Gamma$ for some finite subgroup $\Gamma$ of $SO(n)$ acting freely on $\mathbb{S}^{n-1}$. Let $\beta \in \left(\frac{n-2}{2},n-2\right)$ and $\alpha \in (0,1)$. Then $\lambda^0_{\ALE}(g) = \mathcal{F}_{\ALE}(u_\infty,g)$ (which might be infinite), where $u_\infty$ is the unique, positive solution to the equation
    \begin{equation}
        \begin{cases}
            -4\Delta_g u_\infty + \scal_g u_\infty = 0 \\ u_\infty - 1 \in \C^{2,\alpha}_{-\beta}(M).
        \end{cases}
    \end{equation}
    Moreover, there exists $c>0$ such that $c\leqslant u_\infty\leq 1$. 

\end{proposition}
One issue with Haslhofer's \(\lambda\)-functional is that it is \emph{not} continuous or bounded in $\C^{2,\alpha}_{-\beta}$, i.e. if a sequence of metrics \(g_n\) converges to a metric \(g\) in \(\C^{2,\alpha}_{-\beta}\), there is no guarantee that $\lambda^0_{\ALE}(g_n) \r \lambda^0_{\ALE}(g)$; see, for instance, \cite[Example 3.1]{deruelle-ozuch-2020}. This motivates the definition of the $\lambda_{\ALE}$-functional, which remedies this issue by subtracting from $\lambda^0_{\ALE}(g)$ the mass of the metric.

\begin{definition}[$\lambda_{\ALE}$, \cite{haslhofer}, \cite{deruelle-ozuch-2020}]
    Let $(M^n,g)$ be an ALE metric with $\scal_g\geqslant 0$ and $\scal_g\in L^1(dV_g) $ for $\beta > \frac{n-2}{2}$. We define
    \begin{equation}
        \lambda_{\ALE}(g):=\lambda^0_{\ALE}(g) - \mathfrak{m}(g).
    \end{equation}
\end{definition}
Although $\lambda_{\ALE}$ is originally defined with $\scal_g\in L^1(dV_g) $, by \cite[Theorem 2.17]{baldauf2022spinors}, it extends analytically to ALE metrics of order $\beta>\frac{n-2}{2}$ and $\scal_g\geqslant 0$. By \cite[Proposition 3.4]{deruelle-ozuch-2020}, $\lambda_{\ALE}$ can also be written as
\begin{equation}\label{eq:defn lambdaALE general}
    \lambda_{\ALE}(g) = \lim_{R\to\infty} \left(\int_{\{r\leqslant R\}} (4|\nabla u_\infty|^2 + \scal_g\, u_\infty^2) dV_g - \int_{\{r=R\}}\left\langle \div_e(h)-\nabla^e \tr_e(h),\nu \right\rangle_e dA_g\right)
\end{equation}
if $g=e+h$ close to infinity for $|\nabla^k_eh| \leqslant C_k r^{-\beta-k}$ for $C_k>0$. The proof of \cite[Proposition 3.4]{deruelle-ozuch-2020} also shows that there is continuity of the right-hand-side of \eqref{eq:defn lambdaALE general} at $R\to\infty$, namely
\begin{equation} \label{eq:lambda-ALE-DO}
    \int_{\{r\leqslant R\}} (4|\nabla u_\infty|^2 + \scal_g\, u_\infty^2) dV_g - \int_{\{r=R\}}\left\langle \div_e(h)-\nabla^e \tr_e(h),\nu \right\rangle_e dA_g = \lambda_{\ALE}(g) + O(R^{n-2\beta-2}).
\end{equation}
The proof also shows that if $g_n \r g$ in $\C^{2,\alpha}_{-\beta}$, then $\lambda_{\ALE}(g_n) \r \lambda_{\ALE}(g)$.

A crucial property of $\lambda_{\ALE}$ that is used to prove Theorem \ref{large-scale-behavior-of-mu} is that it is negative on spin ALE manifolds with suitable groups at infinity and positive scalar curvature. It is stated in \cite{deruelle-ozuch-2020} in dimension $4$, but using \cite{dahl} in place of \cite{nakajima} yields the following result in all dimensions.

\begin{proposition}\label{prop: sign lambda spin}
    Let $(M^n,g)$ be an ALE manifold of order $\beta>\frac{n-2}{2}$ with group at infinity satisfying the assumptions of \cite[Theorem 5.1]{dahl}. If $\scal_g \geqslant 0$, then 
    \begin{equation}\label{eq:lambda negative}
        \lambda_{\ALE}(g)\leqslant 0,
    \end{equation}
    with equality if and only if $g$ admits a parallel spinor.
\end{proposition}
This theorem applies to any spin ALE $4$-manifold $M^4$ with group $SU(2)$ at infinity, our main application. 

\section{Large-scale behavior of Perelman's $\mu$-functional on ALE manifolds} \label{large-scale-behavior}
 Throughout this section, we let \((M^n,g)\) denote an ALE manifold of order \(\beta \in \left(\frac{n-2}{2},n-2\right)\) whose end \(M_\infty\) is diffeomorphic to \((\mathbb{R}^n/\Gamma) \setminus B_{R_0}\). We also assume that $\scal_g\geq 0$. We recall that Perelman's entropy functional and $\mu$-functional are defined as
\begin{equation}
    \mathcal{W}(u,g,\tau) = \int_M \left[\tau(4|\nabla u|^2 + \scal_g u^2) - u^2\log(u^2) - nu^2\right]dV_g,
\end{equation}
\begin{equation}
    \mu(g,\tau) = \inf\left\{\mathcal{W}(u,g,\tau) \st u \in W^{1,2}_0(M) \quad \mathrm{ and } \quad ||u||^2_{L^2(dV_g)}= (4\pi \tau)^{\frac{n}{2}}\right\}.
\end{equation}
This is well-defined on ALE metrics by \cite{ozuch-cones}, and we note that $\mu(\R^n/\Gamma):=\mu(\R^n/\Gamma,\tau)=-\log(|\Gamma|)<0$ for all $\tau>0$. Indeed, the Gaussian $\Gamma_\tau$ satisfies $\int_{\R^n/\Gamma} \Gamma_\tau^2 dV_e=\frac{(4\pi\tau)^{\frac{n}{2}}}{|\Gamma|}$, so $\int_{\R^n/\Gamma} |\Gamma|\Gamma_\tau^2dV_e = (4\pi\tau)^{\frac{n}{2}}$ and
\begin{equation}
    \mu(\R^n/\Gamma,\tau)=\mathcal{W}(\Gamma_\tau^2,g_e,\tau) - |\Gamma|\log(|\Gamma|)\int_{\R^n/\Gamma}\Gamma_\tau^2 (4\pi \tau)^{-\frac{n}{2}}dV_e = -\log(|\Gamma|).
\end{equation}

\begin{definition}
    We define $\mu_{\ALE}$ as
\begin{equation} \label{eq:mu-ALE-defn}
    \begin{aligned}
        \mu_{\ALE}(g,\tau) := \inf\left\{\mathcal{W}(u,g,\tau) \st u\in W^{1,2}_0(M) \quad \mathrm{ and } \quad||u||^2_{L^2(dV_g)} = \frac{(4\pi \tau)^{\frac{n}{2}}}{|\Gamma|} \right\}.
    \end{aligned}
\end{equation}
\end{definition}

For instance, if $(M,g) = (\R^n/\Gamma,g_e)$, then the compatibility condition implies
\begin{equation}
    \int_{\R^n} e^{-f}dV_g = |\Gamma|\int_{\R^n/\Gamma} e^{-f}dV_g = (4\pi \tau)^{\frac{n}{2}}.
\end{equation}
Observe that we can also write $\mu_{\ALE}$ is
\begin{equation} \label{eq:mu-ALE-alternate}
    \mu_{\ALE}(g,\tau) = \frac{1}{|\Gamma|}[\mu(g,\tau)-\mu(\R^n/\Gamma)].
\end{equation}
Indeed, if $\mu(g,\tau)=\mathcal{W}(e^{-f},g,\tau)$, then
\begin{equation}
    \mu_{\ALE}(g,\tau) = \mathcal{W}(e^{-f-\log(|\Gamma|)},g,\tau) = \frac{1}{|\Gamma|}[\mu(g,\tau)-\log(1/|\Gamma|)].
\end{equation}
For convenience of notation, we shall define $\alpha_\tau:= \frac{(4\pi \tau)^{\frac{n}{2}}}{|\Gamma|}$ to be the normalizing constant in the definition of $\mu_{\ALE}$ given in \eqref{eq:mu-ALE-defn}.

\subsection{Perelman's \(\mu\)-functional on approximately Euclidean metrics}

In this subsection, we first prove Theorem \ref{large-scale-behavior-of-mu} in the case of \emph{approximately Euclidean metrics} (as defined below), and in \S\ref{perelmans-mu-on-general-ALE}, we adapt the proof to the case of a general ALE metric.

\begin{definition}[Approximately Euclidean metrics]
    For any \(R>R_0\), we define the associated \emph{approximately Euclidean metric} \(g_R\) by
    \begin{equation}
    g_R = \chi_R g + (1-\chi_R)g_e,
    \end{equation}
    where \(0\leq \chi_R \leq 1\) is a cutoff function supported in \(B_{2R}\) which is identically \(1\) on \(B_R\) and satisfies \(|\chi|,|\nabla \chi| \leq \frac{C}{R}\) for some $C>0$.
\end{definition}
We note that \((M^n,g_R)\) is ALE of order \(\beta\) with the same ALE chart on \(M_\infty\) as \((M^n,g)\). By definition, \(\scal_{g_R}=\scal_g\) on \(B_R\) and \(\scal_{g_R}=\scal_e=0\) on \(M\setminus B_{2R}\). In this subsection, we will simply write $\nabla=\nabla^{g_R}$ and $\langle,\rangle=\langle,\rangle_{g_R}$.

In this subsection, we prove the following result, which controls $\mu_{\ALE}(g_R,\tau)$ for large values of \(\tau\) by $\lambda_{ALE}(g_R)$.

\begin{theorem} \label{mu-ALE-is-negative-for-gR}
    Suppose \((M^n,g)\) is an ALE manifold of order \(\beta>\frac{n-2}{2}\). Then 
    \begin{equation}\label{eq: asympt mu gR}
        \mu(g_R,\tau) \leq \mu(\R^n/\Gamma)+\frac{\tau}{(4\pi \tau)^{\frac{n}{2}}}|\Gamma|\cdot \lambda_{\ALE}(g_R) + O(\tau^{\alpha})
    \end{equation}
for \(\tau\) sufficiently large (depending on $R$), where $\alpha<1-\frac{n}{2}$ is as in Lemma \ref{cutoff-asymptotics}.
\end{theorem}
\begin{remark}
    An early tentative proof of the general result for $g$ was based on the fact that as $R\to\infty$, one has $\lambda_{\ALE}(g_R)\to \lambda_{\ALE}(g)$. Unfortunately, the convergence $\mu_{\mathrm{ALE}}(g_R,\tau)\to\mu_{\mathrm{ALE}}(g,\tau)$ is not obvious, especially since in the proof below, $\tau$ is chosen depending on $R$.
\end{remark}

To prove Theorem \ref{mu-ALE-is-negative-for-gR}, we construct a test function \(u_\tau\) which satisfies the compatibility constraint in \eqref{eq:mu-ALE-defn} and
\begin{equation} \label{eq:W-lambda-ALE-inequality-for-gR}
    \mathcal{W}(u_\tau,g_R,\tau) = \frac{\tau}{(4\pi \tau)^{\frac{n}{2}}}\lambda_{\ALE}(g_R) + O(\tau^\alpha), \quad \alpha < 1-\frac{n}{2}.
\end{equation}

We recall that \(\lambda_{\mathrm{ALE}}(g_R)=\lambda^0_{\mathrm{ALE}}(g_R)\) since \(\mathfrak{m}(g_R)=0\). By \cite[Appendix A]{baldauf-ozuch-2}, for $R$ large enough, a minimizer of $\lambda_{\ALE}(g_R)$ with all the properties listed in Proposition \ref{energy-minimizers-for-positive-scalar-curvature} exists (in fact, minimizers exist along the curve of metrics $(1-t)g + tg_R$). Since the standard Gaussian minimizes the entropy on Euclidean space, we construct a function \(\tilde{u}_\tau\) which is a Gaussian \(\Gamma_\tau = \exp(-|x|^2/8\tau)\) outside some compact set. To make the \(\lambda_{\mathrm{ALE}}\)-term appear in \eqref{eq: asympt mu gR}, we make \(\tilde{u}_\tau\) coincide with the minimizer \(u_\infty\) of \(\lambda_{\mathrm{ALE}}(g_R)\) on a compact set. We then interpolate between this minimizer and the Gaussian on an annulus. More precisely,
\begin{equation} \label{eq:constructed-minimizer-for-gR}
    \tilde{u}_\tau :=
    \begin{cases}
        u_\infty & B_{\tau^\varepsilon} \\ \chi_\tau u_\infty + (1-\chi_\tau)\Gamma_\tau & A(\tau^\varepsilon, 3\tau^\varepsilon) \\ \Gamma_\tau & M\setminus B_{3\tau^\varepsilon},
    \end{cases}
\end{equation}
where \(\varepsilon \in (0,\frac{1}{n+2})\) and \(0\leq \chi_\tau \leq 1\) is a cutoff function supported in \(B_{3\tau^\varepsilon}\) which is identically \(1\) on \(B_{\tau^\varepsilon}\) and satisfies \(|\chi_\tau|,|\nabla \chi_\tau| \lesssim \tau^{-\varepsilon}\). We choose \(\tau \gg (2R)^{1/\varepsilon}\) so that $\scal_{g_R}=0$ outside $B_{\tau^\varepsilon}$. The idea behind this construction is that $\tilde{u}_\tau$ ``almost" solves the PDE satisfied by the minimizers of the entropy.

We need to ensure that $||\tilde{u}_\tau||^2_{L^2(dV_{g_R})}$ is close enough to \(\alpha_\tau\) to be able to interpolate between it and its normalization \(u_\tau := c_\tau \tilde{u}_\tau\) without affecting the estimates that follow. This is guaranteed by the next lemma.

\begin{lemma} \label{normalizing-constant-asymptotics}
    Let \(c_\tau\) be the constant which satisfies \(||c_\tau \tilde{u}_\tau||^2_{L^2(dV_{g_R})} = \alpha_\tau\). Then
    \begin{equation}
        c_\tau^2 = 1 + O(\tau^{\varepsilon n - \frac{n}{2}}).
    \end{equation}
    \begin{proof}    
         The compact and cutoff regions can be estimated using the fact that $\tilde{u}_\tau$ is bounded independently of $\tau$ since $u_\infty$ and $\Gamma_\tau$ are. In particular,
        \begin{equation}
            0 \leq \int_{B_{\tau^\varepsilon}} u_\infty^2 dV_{g_R} \leq C_1\tau^{\varepsilon n} \quad \mathrm{and} \quad 0 \leq \int_{A(\tau^\varepsilon,3\tau^\varepsilon)} \tilde{u}_\tau^2 dV_{g_R} \leq C_2\tau^{\varepsilon n}.
        \end{equation}
        As for the noncompact region, the fact that $\int_{\R^n/\Gamma} \Gamma_\tau^2 dV_e = \alpha_\tau$ implies
        \begin{align}
            \int_{M\setminus B_{3\tau^\varepsilon}} \tilde{u}_\tau^2 dV_g &= \int_{(\R^n/\Gamma) \setminus B_{3\tau^\varepsilon}} \Gamma_\tau^2dV_e \in [\alpha_\tau - C_3\tau^{\varepsilon n},\alpha_\tau].
        \end{align}
        Combining these estimates, we obtain
        \begin{equation} \label{eq:L2-estimate-gR}
            \alpha_\tau- C_3\tau^{\varepsilon n} \leq \int_M \tilde{u}_\tau^2 dV_g \leq \alpha_\tau + (C_1+C_2)\tau^{\varepsilon n}.
        \end{equation}
        Multiplying \eqref{eq:L2-estimate-gR} by $c_\tau^2$ and dividing, we obtain
        \begin{align}
            1-\frac{(C_1+C_2)\tau^{\varepsilon n}}{\alpha_\tau} \leq c_\tau^2 \leq  1 + \frac{C_3\tau^{\varepsilon n}}{\alpha_\tau}.
        \end{align}
        Since $\alpha_\tau = O(\tau^{\frac{n}{2}})$, this yields the desired estimate.
    \end{proof}
\end{lemma}

By Lemma \ref{normalizing-constant-asymptotics}, it suffices to show that \(\tilde{u}_\tau\) satisfies \eqref{eq:W-lambda-ALE-inequality-for-gR} to reach the same conclusion for $u_\tau$ since
\begin{equation} \label{eq:normalization-interpolation}
    \mathcal{W}(u_\tau,g_R,\tau) = c_\tau^2\mathcal{W}(\tilde{u}_\tau,g_R,\tau) - |\Gamma|^{-1}\log(c_\tau^2) = \mathcal{W}(\tilde{u}_\tau,g_R,\tau) + O(\tau^{\varepsilon n - \frac{n}{2}}).
\end{equation}
To estimate \(\mathcal{W}(\tilde{u}_\tau,g_R,\tau)\), we first look at the entropy integral over the compact region \(B_{\tau^\varepsilon}\).

\begin{lemma} \label{compact-asymptotics}
    We have the estimate
    \begin{align}
        \int_{B_{\tau^\varepsilon}} &[\tau(4|\nabla \tilde{u}_\tau|^2 + \scal_{g_R} \tilde{u}_\tau^2) - \tilde{u}_\tau^2\log(\tilde{u}_\tau^2) - n\tilde{u}_\tau^2](4\pi \tau)^{-\frac{n}{2}}dV_{g_R} = \frac{\tau}{(4\pi \tau)^{\frac{n}{2}}}\lambda_{\mathrm{ALE}}(g_R) + O(\tau^{\alpha}),
    \end{align}
    where $\alpha = \max(\varepsilon n - \frac{n}{2}, 1-\frac{n}{2}+\varepsilon(n-2\beta-2))$.
    \begin{proof}
        Recall that \(\tau^\varepsilon \gg  2R\), so \(\scal_{g_R}=0\) outside \(B_{\tau^\varepsilon}\). Then we find
        \begin{align}
            \int_{B_{\tau^\varepsilon}} (4|\nabla \tilde{u}_\tau|^2 + \scal_{g_R} \tilde{u}_\tau^2) dV_{g_R} &= \int_{B_{\tau^\varepsilon}} (4|\nabla u_\infty|^2 + \scal_{g_R} u_\infty^2) dV_{g_R} \\ &= \lambda_{\mathrm{ALE}}(g_R) - \int_{M\setminus B_{\tau^\varepsilon}} 4|\nabla u_\infty|^2 dV_{g_R}. \label{eq:compact-gR-estimate-1}
        \end{align}
        Since \(\nabla u_\infty = O(r^{-\beta-1})\), we estimate
        \begin{equation} \label{eq:gradient-estimate-for-F-ALE-minimizer}
            \int_{M\setminus B_{\tau^\varepsilon}} 4|\nabla u_\infty|^2 dV_{g_R} = O(\tau^{\varepsilon(n-2\beta-2)}).
        \end{equation}
        To estimate the Nash entropy term, observe that since there exists $c>0$ independent on $\tau$ so that $c\leqslant\tilde{u}_\tau^2\leqslant c^{-1}$,
        \begin{equation}
            \int_{B_{\tau^\varepsilon}} (\tilde{u}_\tau^2\log(\tilde{u}_\tau^2) + n\tilde{u}_\tau^2) = O(\tau^{\varepsilon n}).
        \end{equation}
        Combining these estimates, we obtain
        \begin{align}
            \int_{B_{\tau^\varepsilon}} &[\tau(4|\nabla \tilde{u}_\tau|^2 + \scal_{g_R} \tilde{u}_\tau^2) - \tilde{u}_\tau^2\log(\tilde{u}_\tau^2) - n\tilde{u}_\tau^2](4\pi \tau)^{-\frac{n}{2}}dV_{g_R} \\ &= \frac{\tau}{(4\pi \tau)^{\frac{n}{2}}}\lambda_{\mathrm{ALE}}(g) + O(\tau^{1 - \frac{n}{2} + \varepsilon(n-2\beta-2)}) + O(\tau^{\varepsilon n - \frac{n}{2}}).
        \end{align}
    \end{proof}
\end{lemma}

We now deal with the cutoff annulus.

\begin{lemma} \label{cutoff-asymptotics}
    We have the estimate
    \begin{align}
        \int_{A(\tau^\varepsilon,3\tau^\varepsilon)}[\tau(4|\nabla \tilde{u}_\tau|^2 + \scal_{g_R}  \tilde{u}_\tau^2) - \tilde{u}_\tau^2\log(\tilde{u}_\tau^2) - n\tilde{u}_\tau^2](4\pi \tau)^{-\frac{n}{2}}dV_{g_R} = O(\tau^{\alpha}),
    \end{align}
    where $\alpha$ is as in Lemma \ref{compact-asymptotics}.
    \begin{proof}
        Since
        \begin{equation}
            \nabla \tilde{u}_\tau = \chi \nabla u_\infty + (1-\chi)\nabla \Gamma_\tau + \nabla \chi(u_\infty - \Gamma_\tau)
        \end{equation}
        and \(|\nabla u_\infty| = O(r^{-\beta-1})\), \(|\chi|\leq 1\), \(|\nabla \chi| \leq C\tau^{-\varepsilon}\), it follows that on \(A(\tau^{\varepsilon},3\tau^{\varepsilon})\),
        \begin{align} \label{eq:gradient-of-u-tau-gR}
            |\nabla \tilde{u}_\tau| &\leq |\nabla u_\infty| + |\nabla \Gamma_\tau| + C\tau^{-\varepsilon}|u_\infty-\Gamma|.
        \end{align}
        To estimate $|u_\infty-\Gamma_\tau|$, we first observe that on $A(\tau^\varepsilon,3\tau^\varepsilon)$,
        \begin{equation}
            \Gamma_\tau \geq \Gamma_\tau(S_{3\tau^\varepsilon}) = 1 - \frac{3^{2\varepsilon -1}}{8}\tau^{2\varepsilon -1} + O(\tau^{4\varepsilon-2}).
        \end{equation}
        Then using that $u_\infty = 1+O(r^{-\beta})$, 
        \begin{equation}
            u_\infty - \Gamma_\tau \leq \frac{3^{2\varepsilon -1}}{8}\tau^{2\varepsilon -1} + O(\tau^{4\varepsilon-2}) + O(\tau^{-\varepsilon \beta})= O(\tau^{\max(-\varepsilon \beta,2\varepsilon-1)}).
        \end{equation}
        Similarly, since $\Gamma_\tau \leq 1$ and $u_\infty \geq 1-cr^{-\beta}$ on $A(\tau^\varepsilon,3\tau^\varepsilon)$,
        \begin{equation}
            u_\infty - \Gamma_\tau \geq -cr^{-\beta} \geq -c\tau^{-\varepsilon \beta}.
        \end{equation}
        It then follows that
        \begin{equation} \label{eq:minimizer-is-close-to-gaussian}
            \tau^{-\varepsilon}|u_\infty -\Gamma_\tau| \lesssim \tau^{\varepsilon -1} + \tau^{-\varepsilon(\beta+1)} \lesssim \tau^{-\varepsilon(\beta+1)},
        \end{equation}
        where we use that $\varepsilon -1 < -\varepsilon(\beta+1)$. Since $|\nabla u_\infty|=O(r^{-\beta-1})$ and $|\nabla \Gamma_\tau|=O(\tau^{\varepsilon-1})$ on $A(\tau^\varepsilon,3\tau^\varepsilon)$, \eqref{eq:gradient-of-u-tau-gR}  and \eqref{eq:minimizer-is-close-to-gaussian} imply the estimate
        \begin{align}
            |\nabla \tilde{u}_\tau| \lesssim r^{-\beta-1} + \tau^{-\varepsilon(\beta+1)}, \quad \mathrm{hence} \quad |\nabla \tilde{u}_\tau|^2 \lesssim r^{-2\beta-2} + \tau^{-2\varepsilon(\beta+1)}.
        \end{align}
        Integrating over the annulus, we obtain
        \begin{equation}
            \frac{\tau}{(4\pi\tau)^{\frac{n}{2}}}\int_{A(\tau^\varepsilon,3\tau^\varepsilon)} |\nabla \tilde{u}_\tau|^2 dV_{g_R} = O(\tau^{1-\frac{n}{2}+\varepsilon(n-2\beta-2)}).
        \end{equation}
        The Nash entropy term is estimated in the same way as before: since $u_\infty$ and $\Gamma_\tau$ are both bounded below away from zero on the annulus, there is there is a constant $c$ such that $c^{-1}\leq \tilde{u}_\tau \leq c$ on the annulus, hence
        \begin{equation}
             \int_{A(\tau^\varepsilon,3\tau^\varepsilon)} (\tilde{u}_\tau^2\log(\tilde{u}_\tau^2) + n\tilde{u}_\tau^2)(4\pi \tau)^{-\frac{n}{2}} = O(\tau^{\varepsilon n-\frac{n}{2}}).
        \end{equation}
        These two estimates and the vanishing of \(\scal_{g_R}\) on the annulus give the desired result.
    \end{proof}
\end{lemma}

It remains to estimate the integral over the noncompact region.

\begin{lemma}
     We have the estimate
    \begin{align}
        \int_{M\setminus B_{3\tau^\varepsilon}}[\tau(4|\nabla \tilde{u}_\tau|^2 + \scal_{g_R}  \tilde{u}_\tau^2) - \tilde{u}_\tau^2\log(\tilde{u}_\tau^2) - n\tilde{u}_\tau^2](4\pi \tau)^{-\frac{n}{2}}dV_{g_R} = O(\tau^{\varepsilon n -\frac{n}{2}}).
    \end{align}
    \begin{proof}
        We first recall that
        \begin{equation}
            \mathcal{W}(\Gamma_\tau,g_e,\tau) = \int_{\R^n} [4\tau|\nabla \Gamma_\tau|^2 - \Gamma_\tau^2\log(\Gamma_\tau^2)](4\pi \tau)^{-\frac{n}{2}}dV_e  -n = 0,
        \end{equation}
        which implies
        \begin{align}
            \int_{\R^n \setminus B_{3\tau^\varepsilon}} (4\tau|\nabla \Gamma_\tau|^2 - \Gamma_\tau^2\log(\Gamma_\tau^2))(4\pi \tau)^{-\frac{n}{2}} dV_e- n &= -\int_{B_{3\tau^\varepsilon}}(4\tau|\nabla \Gamma_\tau|^2 - \Gamma_\tau^2\log(\Gamma_\tau^2))(4\pi \tau)^{-\frac{n}{2}}dV_e \\ &= -\int_{B_{3\tau^\varepsilon}} \frac{r^2}{2\tau}\Gamma_\tau^2 (4\pi \tau)^{-\frac{n}{2}}dV_e \\ &= O(\tau^{\varepsilon(n+2)-1-\frac{n}{2}}) = O(\tau^{\varepsilon n - \frac{n}{2}}),
        \end{align}
        where we use that $2\varepsilon - 1 < \frac{2}{n+2}-1<0$. It then follows that
        \begin{align}
            &\int_{(\R^n/\Gamma) \setminus B_{3\tau^\varepsilon}} (4\tau|\nabla \Gamma_\tau|^2 - \Gamma_\tau^2\log(\Gamma_\tau^2) - n\Gamma_\tau^2)(4\pi \tau)^{-\frac{n}{2}} dV_e \\ &= \frac{1}{|\Gamma|}\left(\int_{\R^n \setminus B_{3\tau^\varepsilon}} (4\tau|\nabla \Gamma_\tau|^2 - \Gamma_\tau^2\log(\Gamma_\tau^2))(4\pi \tau)^{-\frac{n}{2}} dV_e -n + \int_{B_{3\tau^\varepsilon}} n\Gamma_\tau^2(4\pi \tau)^{-\frac{n}{2}}dV_e\right) =O(\tau^{\varepsilon n -\frac{n}{2}}).
        \end{align}
        Since \(\tilde{u}_\tau = \Gamma_\tau\) and $\scal_{g_R}=0$ on $M\setminus B_{3\tau^\varepsilon}$, the lemma follows.
    \end{proof}
\end{lemma}

Since all the powers of \(\tau\) in the previous three lemmas are strictly less than \(1-\frac{n}{2}\), \(\tilde{u}_\tau\) satisfies \eqref{eq:W-lambda-ALE-inequality-for-gR}, hence \(u_\tau\) does as well. This observation in tandem with \eqref{eq:mu-ALE-alternate} proves Theorem \ref{mu-ALE-is-negative-for-gR}.

\subsection{Perelman's $\mu$-functional on general ALE metrics} \label{perelmans-mu-on-general-ALE}

\subsubsection{Estimates in a radial gauge at infinity} \label{existence-of-radial-gauge}

A subtle point of our proof is that the test function we introduce is constructed in a specific convenient gauge. Indeed, in general coordinates, many terms a priori do not decay fast enough. This best extends the above simpler case of approximately Euclidean metrics.
\\

A straightforward ALE extension of the AE result of \cite{x-ray-transform} provides the existence of a \textit{radial gauge} on any ALE manifold. 
\begin{lemma}[{\cite[Lemma 2.2]{x-ray-transform}}]
    Let $(M,g)$ be an ALE metric which is ALE of order $\beta>1$. 
    
    Then, there exists ALE coordinates of order $\beta$ which are \emph{radial}. More explicitly, there exists a compact $K\subset M$, $\rho>0$ and a diffeomorphism $\Phi:M\backslash K\to (\mathbb{R}^n\backslash B_e(\rho))/\Gamma$ so that if we denote $r$ the distance to zero in $(\mathbb{R}^n/\Gamma,g_e)$,
    \begin{enumerate}
        \item $\Phi^*g = g_e+ h$ with $|\nabla_e^k h|_e = O(r^{-\beta-k})$, 
        \item $h(\partial_r,.) =0$.
    \end{enumerate}
\end{lemma}

Working in a radial gauge is very convenient for us as it implies better than expected decays for many tensors of interest in the integrand of the $\mathcal{W}$-functional we need to control. Without this gauge, terms involving $f$ which grow like $r^2$ are not controlled well enough.

The main result of this subsection is the following estimate.

\begin{proposition} \label{noncompact-asymptotics-with-radial-gauge}
    Consider a metric $g = g_e+h$ with $h(\partial_r,\cdot) = 0$ satisfying $|\nabla_e^kh|_e < C_k r^{-\beta-k}$ for $k$ and $C_k>0$, and $R>\rho$ large enough. Define a function $f$ fixing the weighted volume through the equality 
    \begin{equation}\label{eq: preserved weighted volume}
        e^{-f} dV_g = e^{-\frac{r^2}{4\tau}} dV_e,
    \end{equation}
    which in particular implies $f = \frac{r^2}{4\tau}+\frac{\tr_eh}{2} + O(r^{-2\beta})$, where more precisely, one has for some $C_k>0$, $k\in\{0,1,2,3\}$ independent on $r$ or $\tau$,
    \begin{equation}\label{eq:control f radial}
        r^k\left|\nabla^k\Big( f - \Big(\frac{r^2}{4\tau}+\frac{\tr_eh}{2}\Big) \Big)\right|\leqslant C_k r^{-2\beta}.
    \end{equation}
    Then, we have:
    \begin{equation}\label{eq: estimate integrand close to infty}
    \begin{aligned}
        \int_{\{r\geqslant R\}}&\left(\tau(2\Delta f -|\nabla f|^2+\mathrm{Scal})+f-n\right)\frac{e^{-f}}{(4\pi\tau)^\frac{n}{2}}dV_g = O( \tau^{1-\frac{n}{2}} R^{n-\min(2\beta+2, 3\beta)}).
    \end{aligned}
    \end{equation}
    In particular, for large $R$, this is negligible compared to $\tau^{1-\frac{n}{2}}$ provided $\beta>\max(\frac{n-2}{2},\frac{n}{3})$.
\end{proposition}

The first steps of the proof is to estimate the first and second variations of the integrand, seeing $g$ as a small perturbation of $g_e$. We start with the classical first variation.
\\

    At a general metric $g = g_e+k$ in radial gauge, i.e. $k(\partial_r,\cdot)=0$, and any function $f$, the first variation of $\tau(-2\Delta f +|\nabla f|^2-\scal)-f$ in a radial direction $g' = h$ with $h(\partial_r,\cdot)=0$ and $f'=\frac{\tr_g h}{2}$ is
    $$ \tau\left( \langle h,\Ric+\Hess f\rangle - \div_f\div_fh \right) - \frac{\tr h}{2},$$
    where every operation is with respect to $g=g_e+k$. This can be found in \cite[(3.22)]{lott}, for instance.

    This means that for any $R>\rho$, the first variation of 
    $\int_{\{r\geqslant R\}}\left(\tau(-2\Delta f +|\nabla f|^2-\scal)-f+n\right)\frac{e^{-f}}{(4\pi\tau)^\frac{n}{2}}dV_g $
    is 
    \begin{equation}
        \begin{aligned}
            \frac{\tau}{(4\pi\tau)^\frac{n}{2}}&\int_{\{r\geqslant R\}}\left( \left\langle h,\Ric+\Hess f-\frac{g}{2\tau}\right\rangle - \div_f\div_fh \right)e^{-f}dV_g \\
    &= \frac{\tau}{(4\pi\tau)^\frac{n}{2}}\int_{\{r\geqslant R\}}\left\langle h,\Ric+\Hess f-\frac{g}{2\tau} \right\rangle e^{-f}dV_g -\frac{\tau}{(4\pi\tau)^\frac{n}{2}}\int_{r= R} \div_f(h)(\partial_r) dA_g\\
    &=\frac{\tau}{(4\pi\tau)^\frac{n}{2}}\int_{\{r\geqslant R\}}\left\langle h,\Ric+\Hess f-\frac{g}{2\tau} \right\rangle e^{-f}dV_g 
        \end{aligned}
    \end{equation}
    In the last line, we used that for \textit{any} metric $g$, and \textit{any} function $f$, if $h(\partial_r,\cdot)=0$,
    $$\div^{g}_f(h)(\partial_r) = g^{jk}\nabla_{j}h_{k0}-h(\nabla^{g}f,\partial_r) =0.$$
\begin{remark}
    The first variation vanishes at the Gaussian soliton itself.
\end{remark}

We may now give the formula for the second variation of the integrand. 

\begin{lemma}
    Then, the first variation of 
    $$\frac{\tau}{(4\pi\tau)^\frac{n}{2}}\int_{\{r \geqslant R\}}\left\langle h,\Ric+\mathrm{Hess} f-\frac{g}{2\tau} \right\rangle e^{-f}dV_g$$
    at the Gaussian soliton in the direction $g'=h$ radial and $f' = \frac{\tr h}{2}$ is
    \begin{equation}
        \frac{\tau}{(4\pi\tau)^\frac{n}{2}}\int_{\{r\geqslant R\}}\left\langle h,N_f(h)\right\rangle e^{-f}dV_g,
    \end{equation}
    where $N_f(h):=-\frac{1}{2}\Delta_f h-\div_f^*\div_f(h)$.
\end{lemma}
\begin{proof}
    The variation of $\frac{\tau}{(4\pi\tau)^\frac{n}{2}}\int_M\left\langle h,\Ric+\Hess f-\frac{g}{2\tau} \right\rangle e^{-f}dV_g$ at a soliton is computed by Hall-Murphy in \cite{hall-murphy}: it is $\frac{\tau}{(4\pi\tau)^\frac{n}{2}}\int_M\left\langle h,N_f(h)\right\rangle e^{-f}dV_g$ with
    $$N_f(h):=-\frac{1}{2}\Delta_f h - \mathrm{Rm}(h)-\div_f^*\div_f(h).$$
    At the Gaussian soliton, $\mathrm{Rm}(h)$ vanishes.
\end{proof}

Now, the term $\langle \div_f^*\div_f(h),h\rangle$ decays suitably as $r^{-2\beta-2}$ in radial gauge. On the other hand, this is not necessarily true for the term $\langle\Delta_f h,h\rangle$. We will deal with this term thanks to an integration by parts:
    \begin{equation}\label{eq:IBP Deltaf 2 tensors}
        -\frac{1}{2}\frac{\tau}{(4\pi\tau)^\frac{n}{2}}\int_{\{r\geqslant R\}}\left\langle h,\Delta_f h\right\rangle e^{-f}dV_g = \frac{1}{2}\frac{\tau}{(4\pi\tau)^\frac{n}{2}}\int_{\{r\geqslant R\}}|\nabla h|^2e^{-f}dV_g - \frac{1}{2}\frac{\tau}{(4\pi\tau)^\frac{n}{2}}\int_{\{r= R\}}\left\langle h,\nabla_{\partial_r} h\right\rangle e^{-f}dA_g,
    \end{equation}
where the boundary term involves $\left\langle h,\nabla_{\partial_r} h\right\rangle=O(R^{-2\beta-1})$ integrated over a sphere of volume $R^{n-1}$, so as long as $R$ is large and $\beta>\frac{n-2}{2}$, it will be negligible in our context. Similarly, since $|\nabla h|^2 = O(r^{-2\beta-2})$, the first term of the right hand side is also negligible for large $R$.

We are left with the higher order terms in our expansion.

\begin{lemma}
    The higher order terms decay suitably, namely, for $Q(h,f) =O(r^{-3\beta})$ explicated in the proof,
    \begin{equation}\label{eq:expansion W outside compact}
    \begin{aligned}
        \int_{\{r\geqslant R\}}&\left(\tau(2\Delta f -|\nabla f|^2+\scal)+f-n\right)\frac{e^{-f}}{(4\pi\tau)^\frac{n}{2}}dV_g\\
        =&\, -\frac{1}{2}\frac{\tau}{(4\pi\tau)^\frac{n}{2}}\int_{\{r\geqslant R\}}|\nabla h|^2e^{-f}dV_g- \frac{1}{2}\frac{\tau}{(4\pi\tau)^\frac{n}{2}}\int_{\{r= R\}}\left\langle h,\nabla_{\partial_r} h\right\rangle e^{-f}dA_g \\
        &+\frac{\tau}{(4\pi\tau)^\frac{n}{2}}\int_{\{r\geqslant R\}}\langle \div_f^*\div_f(h),h\rangle e^{-f}dV_g\\
        &+\int_{\{r\geqslant R\}} Q(h,f) e^{-f}dV_g.
    \end{aligned}
    \end{equation}
\end{lemma}
\begin{proof}
    Note first that by construction of $f$, the volume form $e^{-f}dV_g$ is assumed independent of $h$, so we can focus on the other terms.

    In the expression of $Q(h,f)$, we find all of the terms in the expansion of $(2\Delta f -|\nabla f|^2+\scal)+f-n$ at $g = g_e+ h$ which are at least cubic in $h$, where we recall that $e^{-f}dV_g = \Gamma_\tau^2 dV_e$, which more explicitly in coordinates means that $ e^{-f} \sqrt{\det(g_e+ h)} = \Gamma_\tau^2 \sqrt{\det(g_e)}$, hence
    $$ f = \frac{r^2}{4\tau} + \log\left(\sqrt{\frac{\det(g_e+ h)}{\det(g_e)}}\right). $$
    Consequently, as observed before, expanding the determinant term, the linear term in the expansion of $f$ in $h$ is $\frac{\tr_eh}{2}$, and we have the control \eqref{eq:control f radial}. This takes care of the terms coming from the third variations of $f$, which yield the worst estimates in $r^{-3\beta}$.

    


    The third order perturbation of the scalar curvature classically only involves terms of the schematic form $h*\nabla h*\nabla h+h*h*\nabla^2 h = O(r^{-3\beta-2})$. 
    \\

    For the term $|\nabla f|^2$, we have the expression $\nabla_g f = (g_e+h)^{ij}d\left(\frac{r^2}{4\tau} + \frac{\tr h}{2} + O(r^{-2\beta}) \right)$ as well as $|\nabla_g f|_g^2 = (g_e+h)(\nabla_g f,\nabla_g f)$. We find third order terms in $r^{-3\beta-2}$ using the radial condition on $h$. 
    \\

    In order to deal with the Laplacian term, we consider the more convenient combination $\Delta f - |\nabla f|^2 = \Delta_f f$. In coordinates, using the above equality $e^{-f} \sqrt{\det(g_e+ h)} = \Gamma_\tau^2 \sqrt{\det(g_e)}$, $\Delta_f f$ has the expression
    $$\Delta_f f = \frac{1}{\Gamma_\tau^2 \sqrt{\det(g_e)}} \partial_i\left( \Gamma_\tau^2 \sqrt{\det(g_e)} g^{ij}\partial_j f \right), $$
    where the simplification is that $\Gamma_\tau^2 \sqrt{\det(g_e)}$ is independent on $h$. We find that the third and higher order terms in the expansion of $\Delta_f f$ are in $ O(r^{-3\beta})$.
\end{proof}

\subsubsection{Perelman's \(\mu\)-functional on ALE metrics}
Using the radial gauge introduced in \S\ref{existence-of-radial-gauge}, we prove Theorem \ref{large-scale-behavior-of-mu} in full generality by a method analogous to that of Theorem \ref{mu-ALE-is-negative-for-gR} in the sense that we construct a function \(u_\tau\) which satisfies \(||u_\tau||^2_{L^2(dV_g)}=\alpha_\tau\) and an analogue of \eqref{eq:W-lambda-ALE-inequality-for-gR}. Although the natural choice of \(u_\tau\) is (the normalization of) \eqref{eq:constructed-minimizer-for-gR}, this introduces some problems in the setting of a general ALE metric.

Using $f$ defined as above by the property $e^{-f}dV_{g_e+h} = \Gamma_\tau^2dV_e$ close to infinity, we instead define \(\tilde{u}_\tau\) as
\begin{equation} \label{eq:constructed-minimizer-for-general-metric}
    \tilde{u}_\tau :=
    \begin{cases}
        u_\infty & B_{\tau^\varepsilon} \\ \chi_\tau u_\infty + (1-\chi_\tau)e^{-f/2} & A(\tau^\varepsilon, 3\tau^\varepsilon) \\ e^{-f/2} & M\setminus B_{3\tau^\varepsilon},
    \end{cases}
\end{equation}
where $f$ is the function constructed using the radial gauge $h$ in Proposition \ref{noncompact-asymptotics-with-radial-gauge}. The proof of Lemma \ref{normalizing-constant-asymptotics} in the case of an arbitrary ALE metric is almost analogous to that of the $g_R$ metrics.
\begin{lemma} \label{normalization-interpolation-g}
    Let \(c_\tau\) be the constant which satisfies \(||c_\tau\tilde{u}_\tau||^2_{L^2(dV_g)} =\alpha_\tau\). Then
    \begin{equation}
        c_\tau^2 = 1 + O(\tau^{\varepsilon n - \frac{n}{2}}).
    \end{equation}

    \begin{proof}
        Since $\tilde{u}_\tau^2 dV_g = \Gamma_\tau^2 dV_e$,
        \begin{equation}
            \int_{M\setminus B_{3\tau^\varepsilon}} \tilde{u}_\tau^2 dV_g = \int_{(\R^n/\Gamma)\setminus B_{3\tau^\varepsilon}} \Gamma_\tau^2 dV_e \in [\alpha_\tau - C\tau^{\varepsilon n},\alpha_\tau],
        \end{equation}
        so the result follows from computations analogous to those in the proof of Lemma \ref{normalizing-constant-asymptotics}.
    \end{proof}
\end{lemma}

Since $u_\infty$ and $\tilde{\Gamma}_\tau$ are sufficiently close in the cutoff region (as computed explicitly below), we can approximate the energy of $\tilde{u}_\tau$ on this region by the corresponding energy of $u_\infty$.
\begin{lemma} \label{cutoff-asymptotics-for-arbitrary-ALE}
    We have the estimate
    \begin{equation}
        \int_{A(\tau^\varepsilon,3\tau^\varepsilon)} (4|\nabla \tilde{u}_\tau|^2 + \scal_g \tilde{u}_\tau^2) dV_g = \int_{A(\tau^\varepsilon,3\tau^\varepsilon)} (4|\nabla u_\infty|^2 + \scal_g u_\infty^2) dV_g + O(\tau^{\varepsilon(n-2\beta-2)}).
    \end{equation}
    \begin{proof}
        Let \(\tilde{\Gamma}_\tau = e^{-f/2}\). Since $f-\frac{r^2}{4\tau}-\frac{\tr_eh}{2} \in \C^3_{-2\beta}(M)$ by \eqref{eq:control f radial},
        \begin{equation}
            \nabla \tilde{\Gamma}_\tau = (1+O(r^{-\beta}))\nabla \Gamma_\tau + \Gamma_\tau O(r^{-\beta-1}),
        \end{equation}
        so
        \begin{equation}
            |\nabla \tilde{\Gamma}_\tau| \lesssim \tau^{\varepsilon-1} + r^{-\beta-1} \lesssim \tau^{-\varepsilon(\beta+1)} + r^{-\beta-1} \quad \mathrm{and} \quad \tau^{-\varepsilon}|u_\infty - \tilde{\Gamma}_\tau| \lesssim \tau^{\varepsilon -1} + \tau^{-\varepsilon(\beta+1)} \lesssim \tau^{-\varepsilon(\beta+1)},
        \end{equation}
        where we use that \(\varepsilon - 1 < -\varepsilon(\beta+1)\).
        Then
        \begin{equation}
            \tilde{u}_\tau-u_\infty=(1-\chi)(\tilde{\Gamma}_\tau-u_\infty) = O(\tau^{-\varepsilon \beta}),
        \end{equation}
        \begin{equation}
            |\nabla(\tilde{u}_\tau-u_\infty)| \leq |(1-\chi)(\nabla \tilde{\Gamma}_\tau-\nabla u_\infty)| + |\nabla \chi(u_\infty-\tilde{\Gamma}_\tau)| = O(\tau^{-\varepsilon(\beta+1)}).
        \end{equation}
        Then using that $\scal_g = O(r^{-\beta-2})$,
        \begin{align}
            \scal_g\tilde{u}_\tau^2 = \scal_gu_\infty^2 + O(\tau^{-2\varepsilon(\beta+1)}) \quad \mathrm{and} \quad |\nabla \tilde{u}_\tau|^2 = |\nabla u_\infty|^2+O(\tau^{-2\varepsilon(\beta+1)}).
        \end{align}
        Integrating over the annulus gives the desired result.
    \end{proof}
\end{lemma}

Notice that in the proofs of Lemmas \ref{compact-asymptotics} and \ref{cutoff-asymptotics}, we used that \(\scal_{g_R}\) vanished outside a compact set. In general, we can only assume that \(\scal_g = O(r^{-\beta-2})\), and \(\int_M \scal_g dV_g\) may diverge. This is why we need to use the very precise gauge introduced in \S\ref{existence-of-radial-gauge}. In particular, this choice of gauge and Proposition \ref{noncompact-asymptotics-with-radial-gauge} with $R=3\tau^{\varepsilon}$ guarantee the following estimate for the noncompact region.
\begin{lemma} \label{noncompact-asymptotics-for-arbitrary-ALE}
    We have the estimate
    \begin{equation}
     \int_{M\setminus B_{3\tau^\varepsilon}}[\tau(4|\nabla \tilde{u}_\tau|^2 + \scal_g  \tilde{u}_\tau^2) - \tilde{u}_\tau^2\log(\tilde{u}_\tau^2) - n\tilde{u}_\tau^2](4\pi \tau)^{-\frac{n}{2}}dV_g = \frac{\tau}{(4\pi\tau)^{\frac{n}{2}}}\int_{S_{3\tau^\varepsilon}} \langle \nabla^e \tr_e(h),\nu \rangle_e dA_e + O(\tau^{\gamma}),
     \end{equation}
     where \(\gamma = \max(\varepsilon n -\frac{n}{2}, 1-\frac{n}{2} + \varepsilon(n-\min(2\beta +2,3\beta))) < 1-\frac{n}{2}\) given our assumptions.

     \begin{proof}
         Integrating by parts, we obtain
         \begin{align}
             &\int_{M\setminus B_{3\tau^\varepsilon}} (4|\nabla \tilde{u}_\tau|^2 +\scal_g\tilde{u}_\tau^2)dV_g \\ &= \int_{M\setminus B_{3\tau^\varepsilon}} (-4\tilde{u}_\tau \Delta \tilde{u}_\tau + \scal_g\tilde{u}_\tau^2)dV_g + \lim_{r\r \infty}\int_{S_r} 4\langle \nabla \tilde{u}_\tau,\nu \rangle_e \tilde{u}_\tau dA_e - \int_{S_{3\tau^\varepsilon}} 4\langle \nabla \tilde{u}_\tau,\nu \rangle_e \tilde{u}_\tau dA_e. \label{eq:energy-IBP}
         \end{align}
         Notice that $e^{-\frac{r^2}{4\tau}}$ decays to $0$ faster than how any polynomial diverges to infinity, so
         \begin{align} \label{eq:boundary-at-infinity}
             \int_{\{r=R\}} \langle \nabla \tilde{u}_\tau,\nu \rangle_e \tilde{u}_\tau dA_e = \int_{\{r=R\}} O(r)e^{-\frac{r^2}{4\tau}} dA_e = O(R^n)e^{-\frac{R^2}{4\tau}} \xrightarrow{R\r \infty} 0.
         \end{align}
         On the other hand, on $S_{3\tau^\varepsilon}$,
         \begin{align}
             -\langle \nabla \tilde{u}_\tau, \nu \rangle_e = \frac{1}{4}\langle \nabla \tr_eh, \nu \rangle \tilde{u}_\tau + \frac{r}{4\tau}\tilde{u}_\tau + O(r^{-2\beta-1}).
         \end{align}
         Then since $\nabla^e \tr_eh = O(r^{-\beta-1})=O(\tau^{-\varepsilon(\beta+1)})$ and $\tilde{u}_\tau^2 = 1+O(\tau^{2\varepsilon -1})$ on $S_{3\tau^\varepsilon}$,
         \begin{align}
             -\int_{S_{3\tau^\varepsilon}} 4\langle \nabla \tilde{u}_\tau,\nu \rangle_e \tilde{u}_\tau dA_e &= \int_{S_{3\tau^\varepsilon}} \langle \nabla \tr_eh,\nu \rangle \tilde{u}_\tau^2 dA_e + O(\tau^{\max(\varepsilon n-1, \varepsilon(n-(2\beta+2)))}) \\ &=\int_{S_{3\tau^\varepsilon}} \langle \nabla \tr_eh, \nu \rangle dA_e + O(\tau^{\max(\varepsilon n -1, \varepsilon(n-(2\beta+2)))}). \label{eq:boundary-at-tau-epsilon}
         \end{align}
         Plugging \eqref{eq:boundary-at-infinity} and \eqref{eq:boundary-at-tau-epsilon} into \eqref{eq:energy-IBP}, we obtain
         \begin{align}
             &\int_{M\setminus B_{3\tau^\varepsilon}} (4|\nabla \tilde{u}_\tau|^2 +\scal_g\tilde{u}_\tau^2)dV_g  \\ &= \int_{M\setminus B_{3\tau^\varepsilon}} (-4\tilde{u}_\tau \Delta \tilde{u}_\tau + \scal_g\tilde{u}_\tau^2)dV_g + \int_{S_{3\tau^\varepsilon}} \langle \nabla \tr_eh, \nu \rangle dA_e + O(\tau^{\max(\varepsilon n -1, \varepsilon(n-(2\beta+2)))}).
         \end{align}
         The lemma follows from Proposition \ref{noncompact-asymptotics-with-radial-gauge} after adding the Nash entropy term to the integrand.
     \end{proof}
\end{lemma}

Having estimated the compact, cutoff, and noncompact regions, we are ready to prove Theorem \ref{large-scale-behavior-of-mu}.

\bigskip
\noindent
\emph{Proof of Theorem \ref{large-scale-behavior-of-mu}.}

Since $\tilde{u}_\tau = u_\infty + O(\tau^{-\varepsilon \beta})$ on $A(\tau^\varepsilon,3\tau^\varepsilon)$ and $\tilde{u}_\tau=u_\infty$ on $B_{\tau^\varepsilon}$, there is a constant $c>0$ such that $c^{-1}\leq \tilde{u}_\tau \leq c$ on $B_{3\tau^\varepsilon}$, so the Nash entropy integral on $B_{3\tau^\varepsilon}$ is $O(\tau^{\varepsilon n - \frac{n}{2}})$ as before.
Then by Lemma \ref{cutoff-asymptotics-for-arbitrary-ALE},
\begin{align}
    &\int_{B_{3\tau^\varepsilon}}[\tau(4|\nabla \tilde{u}_\tau|^2 + \scal_g \tilde{u}_\tau^2) - \tilde{u}_\tau^2\log(\tilde{u}_\tau^2) - n\tilde{u}_\tau^2](4\pi \tau)^{-\frac{n}{2}} dV_g \\ &= \frac{\tau}{(4\pi\tau)^{\frac{n}{2}}}\int_{B_{3\tau^\varepsilon}}(4|\nabla u_\infty|^2 + \scal_g u_\infty^2) dV_g + O(\tau^{\varepsilon n - \frac{n}{2}}) + O(\tau^{1-\frac{n}{2}+\varepsilon(n-2\beta-2)}).
\end{align}
In radial gauge, $\div_e(h)=0$, so by \eqref{eq:lambda-ALE-DO},
\begin{align}
    &\frac{\tau}{(4\pi\tau)^{\frac{n}{2}}}\left[\int_{B_{3\tau^\varepsilon}} (4|\nabla u_\infty|^2 + \scal_g\, u_\infty^2) dV_g + \int_{S_{3\tau^\varepsilon}}\left\langle \nabla^e \tr_e(h),\nu \right\rangle_e dA_g\right]\\ &= \frac{\tau}{(4\pi\tau)^{\frac{n}{2}}}\lambda_{\ALE}(g) + O(\tau^{1-\frac{n}{2}+\varepsilon(n-2\beta-2)}).
\end{align}
Combining this estimate with Lemma \ref{noncompact-asymptotics-for-arbitrary-ALE}, we obtain
\begin{align}
    \mathcal{W}(\tilde{u}_\tau,g,\tau) = \frac{\tau}{(4\pi\tau)^{\frac{n}{2}}}\lambda_{\ALE}(g) + O(\tau^\gamma).
\end{align}
By Lemma \ref{normalization-interpolation-g}, the same estimate holds for $\mathcal{W}(u_\tau,g,\tau)$. Thus,
\begin{equation}
    \mu_{\ALE}(g,\tau) \leq \mathcal{W}(u_\tau,g,\tau) = \frac{\tau}{(4\pi\tau)^{\frac{n}{2}}}\lambda_{\ALE}(g) + O(\tau^\gamma).
\end{equation}
The theorem now follows from \eqref{eq:mu-ALE-alternate}. \qed

\begin{remark}\label{arbitrary-ALE-metric-remark}
    As mentioned in the introduction, Theorem \ref{large-scale-behavior-of-mu} can be extended to arbitrary ALE metrics in the following way. 
    
    On $B_{\tau^\varepsilon}$, we instead take $\tilde{u}_\tau$ to be any function $v$ such that $v-1\in C^{2,\alpha}_{-\beta}(M)$. If we define the functional $
    \mathcal{G}_{\ALE}$ by
    \begin{equation}
        \mathcal{G}_{\ALE}(v,g):= \lim_{R\r\infty}\left(\int_{\{r\leq R\}} (4|\nabla v|^2 + \scal_g\, v^2) dV_g - \int_{\{r=R\}} \left\langle \div_e(h)-\nabla^e \tr_e(h),\nu \right\rangle_e dA_g\right),
    \end{equation}
    then
    \begin{equation}
    \int_{B_{3\tau^\varepsilon}} (4|\nabla v|^2 + \scal_g\, v^2) dV_g - \int_{S_{3\tau^\varepsilon}} \left\langle \div_e(h)-\nabla^e \tr_e(h),\nu \right\rangle_e dA_g = \mathcal{G}_{\ALE}(v,g) + O(\tau^{\varepsilon(n-2\beta-2)}).
    \end{equation}
    Following the steps of the proof in the $\scal_g\geq 0$ case, we obtain
    \begin{equation}
        \mu_{\ALE}(g,\tau) \leq \frac{\tau}{(4\pi\tau)^{\frac{n}{2}}} \mathcal{G}_{\ALE}(v,g) + O(\tau^\gamma).
    \end{equation}
    By choosing $v$ appropriately, this estimate can be made arbitrarily close to the original estimate \eqref{eq:expansion mu}. For lack of application, we do not attempt to prove the existence of minimizers for any metric, which is likely true. 
\end{remark}

\section{Classification of ALE expanding solitons and ALE ancient flows} \label{classification-of-ALE-expanding-solitons-and-ancient-flows}

We now conclude the proofs of our main Theorems \ref{thm: anc ALE} and \ref{thm: exp ALE}. 

\subsection{ALE expanding soliton}

We first prove the simpler Theorem \ref{thm: exp ALE}. Our main tool is the following result from \cite{Bamler-Chen} adapted to our simpler situation of an ALE expanding soliton: 
\begin{proposition}[{\cite[Proposition 4.4]{Bamler-Chen}}]\label{prop:nu exp}
    Let $(M^n,g)$ be an ALE expanding soliton orbifold of order $\beta>0$. Then, one has the following inequality:
    \begin{equation}\label{eq: ineq nu expander}
       \nu(g)\geqslant \nu(\mathbb{R}^n/\Gamma).
    \end{equation}
\end{proposition}

\begin{proof}[Proof of Theorem \ref{thm: exp ALE}]
    Let $(M^n,g)$ be an expanding ALE soliton, i.e. an expanding soliton whose cone at infinity is flat. Then, the argument of \cite[Theorem 3.3.1]{siepmann-thesis} applies to the function $|\Rm|$ in place of $|\Ric|$ since the asymptotic cone is flat. Consequently, the full curvature tensor decays exponentially fast at infinity. From the construction of coordinates of \cite{Bando-Kasue-Nakajima} (for instance), we may obtain ALE coordinates of arbitrary order $\beta$ for $g$, in particular $\beta>\max(\frac{n-2}{2},\frac{n}{3})$.

    Assume towards a contradiction that $(M^n,g)$ is a non-flat expanding soliton which is ALE of order $\beta>\max(\frac{n-2}{2},\frac{n}{3})$, with group at infinity satisfying the assumptions of \cite[Theorem 5.1]{dahl}. 
    
    Given Theorem \ref{large-scale-behavior-of-mu} and \eqref{eq:lambda negative} in Proposition \ref{prop: sign lambda spin}, we see that for some large $\tau$, $\mu_{\ALE}(g,\tau) < 0$, hence by definition,
    $$ \nu(g)< \nu(\mathbb{R}^n/\Gamma). $$
    This contradicts \eqref{eq: ineq nu expander} in Proposition \ref{prop:nu exp}.
\end{proof}

\subsection{ALE ancient Ricci flows}

We now turn to the proof of Theorem \ref{thm: anc ALE}. It relies on the following result to compare with Proposition \ref{prop:nu exp} in the context of ancient Ricci flows $(g_t)_{t\in(0,T]}$ satisfying the mild technical assumption that $g_t$ has \textit{bounded curvature within each compact time interval}, namely, for all $-\infty<t_1<t_2<T$,
\begin{equation}\label{eq:bdd curv finite time}
    \sup_{M\times [t_1,t_2]}|\Rm_{g_t}|_{g_t} < +\infty.
\end{equation}
This is satisfied by our ALE Ricci flows defined in Definition \ref{def:ALErf}.

\begin{proposition}\label{prop: ineq nu ancient}
    Let $(M,g_t)_{\{-\infty<t\leqslant0\}}$ be an ancient Ricci flow with bounded curvature on compact time-intervals as above. Assume that its tangent flow at $t\to-\infty$ is the Gaussian soliton on $\mathbb{R}^n/\Gamma$ for $\Gamma\subset SO(n)$ in the sense of \cite{bamler-21-1}. Then, one has for any $-\infty<t\leqslant0$,
    \begin{equation}\label{eq: ineq nu ancient}
        \nu(g_t) \geqslant \nu(\mathbb{R}^n/\Gamma).
    \end{equation}
\end{proposition}
\begin{proof}
    The above proposition is not strictly speaking available in the literature in this form, but is known to the expert as a combination of other results. We thank Zilu Ma for explaining this proof to us. The steps are as follows:
    \begin{enumerate}
        \item In \cite[Theorem 2.10]{bamler-21-1}, Bamler shows that the pointed Nash entropy, defined below in \eqref{eq: bamler Nash} in Definition \ref{def:Nash}, is continuous with respect to his $\mathbb{F}$-convergence to the tangent solitons in the \textit{compact} case. It is also monotone.
        \item In the noncompact case, by \cite[Appendix A]{bamler-21-2}, the same result holds assuming that the flow has locally in time bounded curvature in the sense of \eqref{eq:bdd curv finite time}. Consequently, the Nash entropy along the flow is bounded below by the Nash entropy of $\mathbb{R}^n/\Gamma$, see Definition \ref{def:Nash} below. This lower bound is shown in \cite[Theorem 1.7]{chan-ma-zhang-2} to be $\nu(\mathbb{R}^n/\Gamma) = \mu(\mathbb{R}^n/\Gamma,\tau)$ for all $\tau>0$.
        \item Finally, \cite[Theorem 1.1]{chan-ma-zhang-1} shows that for any $t\in(-\infty ,T ] $, $\nu(g_t)$ is also bounded below by the Nash entropy of its tangent soliton, hence by the above point,
        $$ \nu(g_t) \geqslant \nu(\mathbb{R}^n/\Gamma). $$
    \end{enumerate}
\end{proof}

\begin{proof}[Proof of Theorem \ref{thm: anc ALE}]
    The proof is now virtually the same as that of Theorem \ref{thm: exp ALE}.

    Assume that $(M^n,g_t)_{-\infty<t\leqslant 0}$ is an ancient Ricci flow which is ALE of order $\beta>\max(\frac{n}{3},\frac{n-2}{2})$ with group at infinity satisfying the assumptions of \cite[Theorem 5.1]{dahl}. 
    
    Using Theorem \ref{large-scale-behavior-of-mu} and \eqref{eq:lambda negative} in Proposition \ref{prop: sign lambda spin}, we find that if $g_t$ did not have a parallel spinor, then
    $$ \nu(g_t) < \nu(\mathbb{R}^n/\Gamma). $$
    This contradicts \eqref{eq: ineq nu ancient} in Proposition \ref{prop: ineq nu ancient}, hence for every $t\in(-\infty,0)$, $g_t$ admits a parallel spinor. 
\end{proof}

\section{Further directions}\label{sec:further}

\subsection{Dynamical functionals along ALE Ricci flows}

As can be seen from the proof of Proposition \ref{prop: ineq nu ancient}, \textit{dynamical} functionals have become useful in the recent theory of Ricci flows. In this section, we review definitions of such functionals from \cite{hein-naber} which have been instrumental in the theory of \cite{bamler-21-1, bamler-2023,bamler-2021-3}. We then introduce an analogous dynamical $\lambda$-functional and conclude with related open questions.

\subsubsection{A pointed entropy functional}
In \cite{hein-naber}, Hein and Naber introduce a localized version of Perelman's entropy as follows. Given a Ricci flow $(M^n,g(t))$ defined for $t\in [-T,0]$, we associate to each point $(x_0,0) \in M\times \{0\}$ a weighted volume form $dV_{x_0}(t)=H_{x_0}(\cdot,t)dV_{g(t)}$, where $H_{x_0}(x,t)$ is the conjugate heat kernel based at $(x_0,0)$. We also write $H_{x_0}(x,t) = (4\pi|t|)^{-\frac{n}{2}}\exp(-f_{x_0}(x,t))$.

\begin{definition}[Pointed entropy, \cite{hein-naber}]

    The \emph{pointed entropy} at scale $|t|$ based at $x_0$ is defined by
    \begin{equation}
        \mathcal{W}_{x_0}(t) = \mathcal{W}(g(t),f_{x_0}(t),|t|).
    \end{equation}
\end{definition}
The time average of the pointed entropy is the \emph{pointed Nash entropy}:

\begin{definition}[Pointed Nash entropy, \cite{hein-naber}]\label{def:Nash}


    The \emph{pointed Nash entropy} at $x_0\in M$ and $t\in [-T,0)$ is defined by
    \begin{equation}
        \mathcal{N}_{x_0}(t) = \frac{1}{|t|}\int^0_t \mathcal{W}_{x_0}(s)ds = \int_M f_{x_0}(t)dV_{x_0}(t) - \frac{n}{2}.
    \end{equation}
\end{definition}
In \cite{bamler-2021-3}, Bamler defines a more general pointed Nash entropy by
\begin{equation}\label{eq: bamler Nash}
    \mathcal{N}_{x_0,t_0}(\tau) = \int_M f_{x_0}(t_0)(4\pi \tau)^{-\frac{n}{2}} \exp(-f_{x_0}(x,t_0)) dV_g - \frac{n}{2},
\end{equation}
where $\tau = t_0-t$. Evaluating this at $t_0=0$ and $t=t_0$ yields $\mathcal{N}_{x_0,0}(t_0) = \mathcal{N}_{x_0}(t_0)$, so these two formulations of the Nash entropy coincide for ancient Ricci flows.

In the same spirit as \cite{hein-naber}, we control the renormalized energy functional $\lambda_{\ALE}$ by introducing a new \emph{dynamical} functional which is defined using the conjugate heat flow.

\subsubsection{A dynamical \(\lambda\)-functional}

Let \((M^n,g(t))\) be a solution to the Ricci flow on an ALE manifold \(M^n\) with bounded curvature and non-negative, integrable scalar curvature, and fix some \(t_0 \in \R\). We also assume that $g(0)$ has bounded mass, which implies $g(t)$ has bounded mass since mass remains constant under the Ricci flow. Let \(\tau(t)=t_0-t\).

\begin{definition}[A dynamical \(\lambda\)-functional]
    The \emph{dynamical} \(\lambda\)-\emph{functional} \(\lambda^{t_0}_{\mathrm{dym}}\) at \(t_0\) is a function of time defined by
    \begin{equation}
        \lambda^{t_0}_{\mathrm{dym}}(t) := \mathcal{F}(f^{t_0}(t),g(t)) - \mathfrak{m}(g(t)),
    \end{equation}
    where \(f^{t_0}(x,t)\) is the solution to the equation
    \begin{equation} \label{eq:new-heat-flow}
        -\partial_tf^{t_0} = \Delta f^{t_0} - |\nabla f^{t_0}|^2 + \scal
    \end{equation}
    with initial condition \(f^{t_0}(x,t_0)\equiv 0\).
\end{definition}
\begin{remark}
    This definition makes sense without a mass term on compact manifolds as well. It might be interesting to study it on Ricci flows reaching a Ricci-flat metric.
\end{remark}
Suppose \(K(x_0,t_0;\cdot,t) = (4\pi \tau)^{-\frac{n}{2}}e^{-f}\) is the heat kernel of \(M\). Then \(f\) satisfies the equation
\begin{equation} \label{eq:original-heat-flow}
    -\partial_tf = \Delta f - |\nabla f|^2 + \scal - \frac{n}{2\tau}.
\end{equation}
If we start the heat flow at \(1\), then \(t_0,\tau \r \infty\) and \eqref{eq:original-heat-flow} becomes \eqref{eq:new-heat-flow}, motivating our definition of \(\lambda_{\mathrm{dym}}(t)\).

We will often write the equivalent equation
\begin{equation} \label{eq:evolution-of-f}
    \partial_\tau f = \Delta f - |\nabla f|^2 + \scal
\end{equation}
and compute variations with respect to \(\tau\).

It is of interest to note that \(\lambda^{t_0}_{\mathrm{dym}}\) is monotonic. Indeed, by \cite[Section 5.4]{chow-ricci-flow-1}, we have the monotonicity formula
\begin{equation}
    \frac{d}{dt}\lambda^{t_0}_{\mathrm{dym}}(t) = 2\int_M |\ric + \hess_f|^2 e^{-f}dV_g \geq 0,
\end{equation}
where we note that the right-hand side is well defined on ALE metrics of order $\beta>\frac{n-2}{2}$.

Our goal is to control \(\lambda^{t_0}_{\mathrm{dym}}\) with the renormalized functional \(\lambda_{\mathrm{ALE}}\) introduced in \cite{deruelle-ozuch-2020}. To this end, we begin by establishing the decay of the solution \(f\) to \eqref{eq:new-heat-flow}.

\begin{lemma} \label{u-upper-bound}
    The solution \(f\) to \eqref{eq:new-heat-flow} is non-negative, or equivalently, \(u:=e^{-f}\leq 1\).

    \begin{proof}
        We first note that
        \begin{equation} \label{eq:evolution-of-u}
            \partial_\tau u = -(\partial_\tau f)e^{-f} = \Delta u -\scal\, u.
        \end{equation}
        Using this and \(\scal \geq 0\),
        \begin{align}
            (\partial_\tau - \Delta)u^2 &= 2u\partial_\tau  u - (2u\Delta u + 2|\nabla u|^2) = -(2\scal \,u^2 + 2|\nabla u|^2) \leq 0.
        \end{align}
        Since $u(0)\leq 1$, the maximum principle implies \(u(\tau) \leq 1\) for all \(\tau\).
    \end{proof}
\end{lemma}
We now show that the decay of \(f\) is preserved in time. It is important to note that although the forward heat flow is considered in \cite{yu-li}, the arguments from \cite[Section 2]{yu-li} we use in this section only depend on the assumption that $(M^n,g(t))$ is a solution to the Ricci flow with bounded curvature on $M \times [0,\infty)$.

\begin{proposition} \label{decay-of-f}
    For all \(\tau\), there is a constant \(C(\tau)>0\) such that \(f^{t_0}(\tau) \leq C(\tau)r^{-2-\beta}\) for \(r\) sufficiently large.

    \begin{proof}
        By \cite[Theorem 2.2]{yu-li}, there is a constant \(C_1\) independent of \(t\) such that \(\scal(\tau) \leq C_1r^{-2-\beta}\). Then by \eqref{eq:evolution-of-f},
        \begin{equation}
            (\partial_\tau - \Delta)f^{t_0} \leq \scal \leq C_1r^{-2-\beta}.
        \end{equation}      
        Define \(h:=r^{2+\beta}\) and \(w:=hf^{t_0}\). Then
        \begin{align}
            (\partial_\tau - \Delta)w &= h(\partial_\tau - \Delta)f^{t_0} - 2\langle \nabla h,\nabla f^{t_0} \rangle - f^{t_0}\Delta h\\ &\leq C_1 - 2\langle \nabla h,\nabla f^{t_0} \rangle - f^{t_0}\Delta h \\ &= C_1 - 2\langle \nabla(\log h),\nabla w \rangle + Bw,
        \end{align}
        where \(B:= \frac{2|\nabla h|^2-h\Delta h}{h^2}\). By \cite[Theorem 2.2]{yu-li}, \(|B(\tau)|\leq C_2\) for some \(C_2\). Define \(G(x,\tau):=C_1+C_2x\). Then since \(w\geq 0\),
        \begin{equation}
            (\partial_\tau - \Delta)w \leq G(w,\tau) - 2\langle \nabla(\log h),\nabla w \rangle.
        \end{equation}
        Since \(w(0)=0\), 
        \begin{equation}
            C(\tau) = \frac{C_1}{C_2}\left(e^{C_2\tau}-1\right) \quad \mathrm{solves} \quad \begin{cases}
                U'(\tau) &= C_1+C_2U \\ U(0) &= w(0).
            \end{cases}
        \end{equation}
        We now have
        \begin{equation}
            Lw := (\partial_\tau - \Delta)w - \langle X(\tau),\nabla w \rangle - G(w,\tau) \leq 0,
        \end{equation}
        where \(X(\tau)=-2\nabla(\log h)\). It now follows from the maximum principle that \(w(\tau) \leq C(\tau)\). Then \(f^{t_0}(\tau) \leq C(\tau)r^{-2-\beta}\), as desired.
    \end{proof}
\end{proposition}

\begin{definition}[The $\lambda^\infty_{\mathrm{dym}}$-functional]
    We define the functional $\lambda^\infty_{\mathrm{dym}}$ by
    \begin{equation}
        \lambda^\infty_{\mathrm{dym}}(t) := \liminf_{t_0 \r \infty} \lambda^{t_0}_{\mathrm{dym}}(t) = \liminf_{t_0 \r \infty}\mathcal{F}(f^{t_0}(t),g(t)) - \mathfrak{m}_{\ADM}(g).
    \end{equation}
\end{definition}
In the above definition, we may omit the reference to a specific time $t$ in $\mathfrak{m}_{\ADM}(g(t))$ since the mass remains constant under the Ricci flow.

\begin{corollary}
    The $\lambda^\infty_{\mathrm{dym}}$-functional dominates $\lambda_{\ALE}$ in the sense that
    \begin{equation}
        \lambda^\infty_{\mathrm{dym}}(t) \geq \lambda_{\ALE}(t).
    \end{equation}
    \begin{proof}
        By Proposition \ref{decay-of-f}, $e^{-f^{t_0}(t)} = 1 + O(r^{-\beta-2})$. Also, since $\lambda^{t_0}_{\mathrm{dym}}(t_0)(t_0) = \int_M \scal_gdV_g - \mathfrak{m}(t_0)$ is well-defined in $\C^{2,\alpha}_{-\beta}$ and we integrate $2\int_M |\ric+\hess_f|^2 e^{-f}dV_g$ for other times, $|\nabla e^{-f/2}|^2$ is integrable. Thus, $\lambda^{t_0}_{\mathrm{dym}}(t) \geq \lambda_{\ALE}(g(t))$ by definition. Taking the limit infimum as $t_0 \r \infty$ gives the desired inequality.
    \end{proof}
\end{corollary}

\subsection{Open questions}
\subsubsection{A lower bound for the $\mu$-functional on ALE manifolds}
In Section \ref{large-scale-behavior}, we only prove that $\mu_{\ALE}(g,\tau)$ is bounded above by a quantity asymptotic to $\frac{\tau}{(4\pi \tau)^{\frac{n}{2}}}\lambda_{\ALE}(g)$, not that $\mu(g,\tau)$ is exactly asymptotic to $\frac{\tau}{(4\pi \tau)^{\frac{n}{2}}}\lambda_{\ALE}(g)$. 

Additionally, we recall that in Section \ref{large-scale-behavior}, we require that $\beta > \frac{n}{3}$ in order to apply Proposition \ref{noncompact-asymptotics-with-radial-gauge}. We invite the reader to generalize the result so that we can simply assume $\beta > \frac{n-2}{2}$, although we would ultimately like to get rid of these hypotheses of decay at infinity.

\begin{question}
    Let $(M,g)$ be ALE of order $\beta>\frac{n-2}{2}$ with nonnegative scalar curvature. Do we have 
    \begin{equation}\label{eq: ques asympt mu}
        \mu(g,\tau) = \mu(\mathbb{R}^n/\Gamma) + \frac{\tau}{(4\pi \tau)^{\frac{n}{2}}} \lambda_{\mathrm{ALE}}(g) + O(\tau^{\gamma}), \quad \gamma<1-\frac{n}{2}?
    \end{equation}
    What is the expansion of $\mu(\tau)$ if one only has $\beta>0$?
\end{question}

\begin{question}
    Are all ancient Ricci flows with tangent soliton $\mathbb{R}^n/\Gamma$ of order $\beta>\frac{n-2}{2}$?
\end{question}

\subsubsection{An asymptotic description of the minimizers of $\mu$ at large scales}
A previous attempt at proving Theorem \ref{large-scale-behavior-of-mu} involved finding an asymptotic description of the minimizers of $\mu(g,\tau)$ for large values of $\tau$; however, it is not clear to us how these minimizers should behave. These minimizers should be asymptotic to the standard Euclidean Gaussians on ALE metrics with $\Gamma\neq \{\operatorname{Id}\}$ and they should approach the minimizers $u_\infty$ in a compact part of the manifold, but the transition region is poorly controlled. It is unclear if these Gaussians should be ``centered'' in the AE case; the center of the Gaussian could drift to infinity. At the very least, it is known that these minimizers are exponentially decaying (see \cite[Theorem 2.3]{zhang}, for instance).

\subsubsection{Asymptotics of dynamical functionals}

Lastly, using the above dynamical counterparts of the functionals $\mu$ and $\lambda_{\ALE}$, we ask if one can hope to obtain a \textit{dynamical} analogue of \eqref{eq: ques asympt mu}.

\begin{question}
    Do the functionals $\lambda^{t_0}_{\mathrm{dym}}$ control the asymptotics of the pointed Nash entropy at large scales $|t|\gg 1$ along an ALE Ricci flow?
\end{question}
This question could be asked about both immortal and ancient Ricci flows on ALE spaces, for instance.

\newpage 
\bibliographystyle{alpha}
\bibliography{Refs2}
\nocite{*}

\end{document}